\documentclass[11pt, reqno]{amsart}
\usepackage{amsfonts}
\usepackage{amsmath}
\usepackage{amssymb}
\usepackage{amsthm}
\usepackage{enumerate}
\usepackage{mathdots}
\usepackage{hyperref}
\usepackage{caption}
\usepackage{xcolor}
\usepackage{adjustbox}
\usepackage{bbold}
\usepackage{mathtools}
\usepackage{hyperref}
\newtheorem{theorem}{Theorem}[section]
\newtheorem{lemma}[theorem]{Lemma}
\newtheorem{proposition}[theorem]{Proposition}
\newtheorem{definition}[theorem]{Definition}
\newtheorem{cor}[theorem]{Corollary}
\newtheorem{prob}[theorem]{Problem}
\newtheorem*{rmk*}{Remark}

\newcommand{\md}[1]{\ensuremath{(\mbox{mod}\, #1)}}
\newcommand{\mdsub}[1]{\ensuremath{(\mbox{\scriptsize mod}\, #1)}}
\newcommand{\mdlem}[1]{\ensuremath{(\mbox{\textup{mod}}\, #1)}}
\newcommand{\mdsublem}[1]{\ensuremath{(\mbox{\scriptsize \textup{mod}}\, #1)}}
\DeclareMathOperator{\Span}{span}

\DeclareMathOperator\supp{supp}

\DeclareMathOperator{\Proj}{Proj}
\usepackage{fancyhdr}

\title{Large sum-free subsets of sets of integers via $L^1$-estimates for trigonometric series}
\author{Benjamin Bedert}
\begin{document}

\begin{abstract}
A set $B$ is said to be \emph{sum-free} if there are no $x,y,z\in B$ with $x+y=z$. We show that there exists a constant $c>0$ such that any set $A$ of $n$ integers contains a sum-free subset $A'$ of size $|A'|\geqslant n/3+c\log \log n$. This answers a longstanding problem in additive combinatorics, originally due to Erd\H{o}s.

\end{abstract}

\maketitle
\tableofcontents
\section{Introduction}
A set $B$ is said to be sum-free if it contains no three elements $x,y,z$ with $x+y=z$. The study of sum-free sets can be
traced back to Schur \cite{schur}, who introduced the concept in proving that Fermat’s last theorem does not hold in $\mathbf{Z}/p\mathbf{Z}$. There is a substantial amount of literature on sum-free sets, most of which we will not discuss here; the interested reader may consult the survey \cite{taovu} of Tao and Vu. One of the most central open questions in this area is that of determining the quantity $S(N)$ which is defined to be the largest integer $S$ such that any set of $N$ positive integers contains a sum-free subset of size $S$. Let us write $S(A)$ for the size of the largest sum-free subset of $A$, so that
\begin{equation}\label{S(A)defi}
    S(N)\vcentcolon= \min_{A\subset\mathbf{N}:|A|=N}S(A).
\end{equation}
The following classical argument of Erd\H{o}s \cite{erdos} proves that $S(A)\geqslant |A|/3$ for any $A\subset\mathbf{Z}\setminus\{0\}$. Let $\mathbf{T}=\mathbf{R}/\mathbf{Z}$ and note that the interval $(1/3,2/3)\subset\mathbf{T}$ is sum-free. Hence, for any $x\in \mathbf{T}$ the set $A_x=\{a\in A: ax\md 1\in (1/3,2/3)\}$ is sum-free. One can conclude by observing that $\mathbb{E}_{x\in\mathbf{T}}|A_x|=\sum_{a\in A}\mathbb{P}(ax\md1\in (1/3,2/3))=|A|/3$ which shows that one of the sum-free sets $A_x$ must have size at least $|A|/3$. 

\medskip

Despite its simplicity, only minor improvements over this lower bound have been obtained. Alon and Kleitman \cite{alonkleitman} observed that Erd\H{o}s's argument may in fact be improved to give $S(N)\geqslant (N+1)/3$. The best bound $S(N)\geqslant (N+2)/3$ before this work was established in a celebrated paper of Bourgain \cite{bourgain} using an elaborate Fourier analytic approach. Recent work of Shakan \cite{shakan} provides an alternative proof of Bourgain's bound using a similar method. It has been a long-standing open problem to find a more substantial improvement over Erd\H{o}s's lower bound and this question appears in the work of various authors such as \cite{alonkleitman, bourgain, eberhardgreenmanners,erdos, guy, jingwu1, jingwu2,lewko,shakan,taovu,taovubook}. The main problem is to prove the following widely believed estimate for $S(N)$ which asserts that one can improve the Erd\H{o}s-Alon/Kleitman-Bourgain bounds by an arbitrarily large constant. This problem is also listed as Problem 1 on Green's list \cite{green} of 100 open problems.
\begin{prob}\label{conj:sumfree}
    Is there a function $\omega(N)\to\infty$ such that $S(N)\geqslant \frac{N}{3}+\omega(N)$?
\end{prob}  
Our aim in this paper is to establish the following theorem confirming this. 
\begin{theorem}\label{th:main}
    There exists some constant $c>0$ such that for all finite sets $A\subset\mathbf{Z}$ we have $S(A)\geqslant \frac{|A|}{3}+c\log \log |A|$. In particular, $S(N)\geqslant \frac{N}{3}+c\log\log N$.
\end{theorem}

We further prove a strong structural result for sets where $S(A)\leqslant N/3+C$ which provides non-trivial information about the global structure of $A$, even for a much larger value of $C$ than $\log\log N$.
\begin{theorem}[99\% Structure Theorem]\label{th:mainstructure}
    Let $A\subset\mathbf{Z}\setminus\{0\}$ be a set of size $N$ with $S(A)\leqslant N/3+C$. Then $A$ has a Freiman-isomorphic copy inside $[-N^{C^{O(1)}},N^{C^{O(1)}}]$.\footnote{By a Freiman isomorphism, we mean an $F_4$-isomorphism as in Definition \ref{def:F_4isomo}.} Moreover, for any parameter $K>1$, we can find a partition
\begin{equation}
A=\left(\bigcup_{j=1}^sA_j\right)\cup B
\end{equation}
which has the following properties.
\begin{itemize}
    \item [(i)] For each $j\in[s]$, the set $A_j$ has size $|A_j|\gg (KC)^{-O(1)}N$ and small doubling $|A_j-A_j|\leqslant (CK)^{O(1)}|A_j|$. In particular, $A_j$ is contained in a generalised arithmetic progression $P_j$ of dimension $d_j\ll (KC)^{O(1)}$ and of size $|P_j|\ll e^{(KC)^{O(1)}}|A_j|$.
    \item [(ii)] The set $B$ is small: $|B|\ll (KC)^{-10}N$.
\end{itemize}
\end{theorem}
Though not the topic of study in this paper, we briefly discuss progress on the upper bounds for $S(N)$. This bound has been improved many times; we summarise this in the following table.
\begin{center}
\begin{tabular}{ |c|c| } 
\hline
Author & Value of $c$ s.t. $S(N)\leqslant cN+o(N)$ \\
\hline
Hinton \cite{erdos} & $7/15$ \\
Klarner \cite{erdos} & $3/7$ \\
Alon, Kleitman \cite{alonkleitman} & $12/29$ \\
Malouf, Furedi \cite{malouf,guy} & $2/5$ \\
Lewko \cite{lewko} & $11/28$ \\
Alon \cite{alon} & $11/28-\varepsilon$ \\
Eberhard, Green and Manners \cite{eberhardgreenmanners} & $1/3$ \\
\hline
\end{tabular}
\end{center}
The first five of these bounds were obtained using increasingly better explicit constructions of sets $A$ for which $S(A)\leqslant c|A|$, where $c$ is the corresponding constant in the bounds above. The breakthrough paper of Eberhard, Green and Manners \cite{eberhardgreenmanners} establishes an upper bound $S(N)\leqslant \frac{N}{3}+o(N)$ which matches Erd\H{o}s's lower bound up to a function which is $o(N)$. Their proof employs an elegant argument based on the arithmetic regularity lemma which leads to a more-or-less ineffective bound for $o(N)$; determining a reasonable upper bound remains an interesting problem. Contrary to the arguments that came before, the sets $A$ with $S(A)\leqslant \frac{|A|}{3}+o(|A|)$ whose existence is proved by Eberhard, Green and Manners are not explicitly constructible, but Eberhard \cite{eberhard} later provided explicit examples of such sets. 

\section{Setup and paper overview}

We begin by discussing a well-known strategy for obtaining lower bounds for $S(N)$, dating back (in its simplest form) to the work of Erd\H{o}s \cite{erdos}. Let $A\subset\mathbf{Z}$ have size $N$. By simply removing $0$ if it lies in $A$, it is sufficient to establish Theorem \ref{th:main} for sets $A\subset\mathbf{Z}\setminus\{0\}$ so we assume throughout the rest of this paper that $0\notin A$. We write $\mathbf{T}=\mathbf{R}/\mathbf{Z}$ for the one-dimensional torus and recall that the interval $(1/3,2/3)\subset\mathbf{T}$ is sum-free. Hence, for any $x\in \mathbf{T}$ the set $\{a\in A: ax\md 1\in (1/3,2/3)\}$ is sum-free and we deduce the following important basic estimate
\begin{equation*}\label{basicestimate1}
    S(A)\geqslant\max_{x\in \mathbf{T}}\sum_{a\in A}\phi(ax)=\frac{N}{3}+\max_{x}\sum_{a\in A}(\phi-1/3)(ax),
\end{equation*}
where $\phi$ is the characteristic function of the interval $(1/3,2/3)$. Noting that $$\max_x\sum_{a\in A}(\phi-1/3)(ax)\geqslant\int_\mathbf{T}\sum_{a\in A}(\phi-1/3)(ax)\,dx=|A|\int_0^1\phi(x)\,dx-|A|/3=0$$ recovers Erd\H{o}s's bound and Bourgain obtained his improvement by showing that $\max_x\sum_{a\in A}(\phi-1/3)(ax)>1/3$ using the Fourier-theoretic properties of $\phi-1/3$. 

\medskip

Before stating our main theorem, we introduce the following strong notion of Freiman isomorphism.
\begin{definition}\label{def:F_4isomo}
Let $G,G'$ be Abelian groups and let $A\subset G$, $A'\subset G'$. We say that a map $\phi:A\to A'$ is an \textit{$F_4$-homomorphism} if whenever $a_i\in A$ and $\varepsilon_i\in\{-1,0,1\}$ for $i\in[4]$ satisfy $$\sum_{i=1}^4\varepsilon_ia_i=0,$$
then $$\sum_{i=1}^4\varepsilon_i\phi(a_i)=0.$$ We say that $A$ and $A'$ are \textit{$F_4$-isomorphic} if there is a bijective $F_4$-homomorphism $\phi:A\to A'$ so that $\phi^{-1}$ is also an $F_4$-homomorphism.
\label{frei}
\end{definition}

It is obvious that $S(A)=S(A')$ for $F_4$-isomorphic sets. Our main result is to prove the following theorem which clearly implies Theorem \ref{th:main}.
\begin{theorem}
    Let $A\subset\mathbf{Z}\setminus\{0\}$. Then there exists a set $B\subset\mathbf{Z}\setminus\{0\}$ which is $F_4$-isomorphic to $A$ and which satisfies
    $$\max_{x}\sum_{b\in B}\left(\phi-\frac{1}{3}\right)(bx)\gg\log\log |B|.$$
    
\end{theorem}
Before beginning the proof, we attempt to give a broad outline of our approach and indicate where it differs from previous work due to Bourgain \cite{bourgain}. Bourgain proved that $S(N)\geqslant (N+2)/3$, but perhaps the most interesting part of his paper is his progress for $(3,1)$-sum-free sets, which he defines to be sets containing no $x,y,z,w$ with $x+y+z=w$. In analogy to $S(N)$, the quantity $S_{(3,1)}(N)$ is then defined to be the largest number so that any set of $N$ positive integers contains a $(3,1)$-sum-free subset of size $S_{(3,1)}(N)$. Noting that the interval $(1/8,3/8)\subset \mathbf{T}$ is $(3,1)$-sum-free allows one to use Erd\H{o}s's argument and obtain $S_{(3,1)}(N)\geqslant N/4$. Bourgain proved the substantially better bound $S_{(3,1)}(N)\geqslant N/4+(\log N)^{1-o(1)}$. However, his argument rather crucially relies on the existence of another maximal $(3,1)$-sum-free subinterval of $\mathbf{T}$, namely $(5/8,7/8)=-(1/8,3/8)$, which allows him to combine the Fourier series of $1_{(1/8),(3/8)}$ and $1_{(5/8,7/8)}$ to get what is essentially a one-sided Fourier series (i.e. its Fourier spectrum consists of non-negative integers only). As Bourgain points out, such an approach cannot be applied to bound $S(N)$ since it is easy to see that $(1/3,2/3)$ is the unique sum-free interval in $\mathbf{T}$ with measure $1/3$. Recently, Jing and Wu \cite{jingwu1, jingwu2} extended Bourgain's method and showed that $S_{(k,\ell)}(N)\geqslant N/(k+\ell)+(\log N)^{1-o(1)}$ for various other pairs $(k,\ell)$, perhaps most interestingly $(2,4)$ and $(1,5)$, but their method still relies on the existence of asymmetric maximal $(k,\ell)$-sum-free subset of the torus. We also mention that Eberhard \cite{eberhard} has shown that $S_{(k,1)}(N)\leqslant N/(k+1)+o(N)$ and that Jing and Wu proved the analogous bound for all $(k,\ell)$.

\medskip

\textbf{Overview of the paper.}
In Section 4, we begin by recalling some of Bourgain's approach which considers the Fourier expansion $$F_A(x)=\sum_{a\in A}(\phi-1/3)(ax)=\sum_{a\in A}\sum_{n\geqslant 1}\frac {\chi(n)}{n}\cos 2\pi nax,$$ where $\chi$ is a character mod 3. Bourgain shows in particular that in order to establish Theorem \ref{th:main}, it suffices to prove that $\lVert F_A\rVert_1\gg \log \log N$. Bourgain also observed that one can `sift' out the contribution of $n>1$ and that a bound $\lVert F_A\rVert_1\gg C$ would follow if $\lVert \hat{1}_A\rVert_1\gg C\log N$.

\smallskip

The first step in our argument is to establish two inverse theorems describing structural properties of sets $A$ for which $\lVert \hat{1}_A\rVert_1\ll C\log N$ and $C$ is `small'. The main result in Section 5 shows that such sets $A$ have small additive dimension $\dim(A)$. In Section 6, we exploit such additive information about $A$ to find a `dense' Freiman-isomorphic copy $B$ of $A$ and we use this to obtain strong bounds on the Fourier coefficients and $L^2$ norm of (certain modifications of) $F_{B}$.  

\smallskip

Section 7 is concerned with studying the distribution of $A$ in residue classes modulo powers of `small' primes $p\leqslant (\log N)^{1/2}$. First, we use a combinatorial argument to show that a large part of $A$ must lie in a single such residue class when $\dim(A)$ is small. Secondly, we prove in Proposition \ref{prop:distmodpstruct} that either $\lVert F_A\rVert_1\gg \log \log N$ or else that the distribution of $A$ in these residue classes is highly structured. The final step in the proof of Theorem \ref{th:main} is accomplished in Section 8 and consists of exploiting this non-Archimedean structure in $A$ to construct an explicit test function $\Phi$ which witnesses that $\lVert F_A\rVert_1$ is large (i.e. we construct $\Phi$ s.t. $|\Phi|\leqslant 1$ and $\int_0^1 F_A(x)\Phi(x)\,dx\gg \log \log N$).

\smallskip

The proof of the structural result Theorem \ref{th:mainstructure} is given in Section 9 and proceeds by bootstrapping our application of the inverse results from Section 5. 

\medskip

\textbf{Acknowledgements.}
The author would like to thank Thomas Bloom, Ben Green and Mehtaab Sawhney for their detailed reading of the paper and for providing many useful suggestions and comments. The author also gratefully acknowledges financial support from the EPSRC.
\section{Notation and prerequisites}
We use the asymptotic notation $f=O(g)$ or $f\ll g$ if there is an absolute constant $C$ such that $|f(x)|\leqslant C g(x)$ for all $x$, and we write $f=o(g)$ if $f(x)/g(x)\to 0$ as $x\to\infty$. 

\smallskip

We use the notation $e(t)=e^{2\pi i t}$. Let $\mathbf{T}=\mathbf{R}/\mathbf{Z}$ be the one-dimensional torus. Throughout the paper, we shall write $c(x)$ for $\cos(2\pi x)$ so that $c(\cdot)$ is a well-defined (1-periodic) function on $\mathbf{T}$. For a suitably  integrable function $g:\mathbf{T}\to\mathbf{C}$ we denote, for $p\in[1,\infty)$, its $L^p$-norm by $$\lVert g\rVert_p\vcentcolon= \left(\int_0^1|g(t)|^p\,dt\right)^{1/p}.$$ Its Fourier transform is $\hat{g}:\mathbf{Z}\to\mathbf{C}$ which is defined by $\hat{g}(n)=\int_\mathbf{T}g(t)e(-nt)\,dt$. If $f:\mathbf{Z}\to\mathbf{C}$ is a function, we shall denote its Fourier transform $\hat{f}:\mathbf{T}\to\mathbf{C}$ by $\hat{f}(x)=\sum_{n\in\mathbf{Z}}f(n)e(nx)$. 

\smallskip

For two functions $g,h\in L^2(\mathbf{T})$ we define $\langle g,h\rangle =\int_\mathbf{T} g(x)\overline{h(x)}\,dx$ and we shall frequently use Parseval's theorem which states that $\langle g,h\rangle =\sum_{n\in\mathbf{Z}}\hat{g}(n)\overline{\hat{h}(n)}$. We also define their convolution $g*h(x) = \int_\mathbf{T}g(x-y)h(y)\,dy$ and note that $\widehat{g*h}(n) = \hat{g}(n)\hat{h}(n)$. 

\smallskip

The following seminal result is known as the ‘Littlewood $L^1$ conjecture', which was proved by McGehee, Pigno and Smith \cite{mcgeheepignosmith} and independently by Konyagin \cite{konyaginl}.
\begin{theorem}[Littlewood's $L^1$ conjecture]
Let $a_1,a_2,\dots,a_k$ be complex numbers and $n_1<n_2<\dots<n_k$ be integers. Then $$\left\lVert \sum_{j=1}^k a_je(n_jt)\right\rVert_1\gg \sum_{j=1}^k\frac{|a_j|}{j}.$$
\label{Litllconj}    
\end{theorem}
\begin{cor}
    Let $B\subset\mathbf{Z}$ be finite, then $\lVert \hat{1}_B\rVert_1\gg \log |B|$.
\end{cor}
We shall not, in fact, use either of these results in this paper. Rather, we will employ various non-trivial modifications of some ideas appearing in the McGehee-Pigno-Smith proof of the Littlewood $L^1$ conjecture. Their proof proceeds, like much of the earlier progress on the $L^1$ conjecture, by constructing a function $\Phi$ such that $|\Phi|\ll 1$ and $\langle \hat{1}_B,\Phi\rangle$ is `large'. The most basic part of their construction of such a test function is something that we will use repeatedly, so we state the following general lemma about constructing these. In fact, the construction in the following lemma already differs from that of McGehee-Pigno-Smith; the reader may notice that we do not require the one-sidedness of the Fourier spectrum (a condition that is crucial for their original construction).
\begin{lemma}[M-P-S basic construction of test functions]\label{lem:basicMPS}
    Let $f:\mathbf{Z}\to\mathbf{C}$ be a function with  finite support $\supp(f)$. Let $X_1,X_2,\dots,X_J\subset\mathbf{Z}$ be finite and define the functions
    \begin{align*}
        g_i:\mathbf{Z}\to\mathbf{C}:g_i(n)=\begin{cases}
            |X_i|^{-1}\frac{f(n)}{|f(n)|} & \text{if $n\in X_i\cap \supp f$,}\\
            0 & \text{otherwise}.
        \end{cases}
    \end{align*}
We further define $Q_i(x)=e^{-|\hat{g}_i(x)|}$ and
\begin{equation*}
    \Phi_j=\hat{g}_j+\hat{g}_{j-1}Q_j+\hat{g}_{j-2}Q_{j-1}Q_{j}\dots+\hat{g}_1Q_2\dots Q_j,
\end{equation*}
and note that these are functions defined on $\mathbf{T}$.
Then the $g_i$ and $\Phi_j$ satisfy the following properties
\begin{align}
\lVert\Phi_j\rVert_\infty&\leqslant 10,\label{basicMPS}\\
    \lVert \hat{g}_i\rVert_\infty&\leqslant 1,\nonumber\\
    \lVert \hat{g}_i\rVert_2&\leqslant |X_i|^{-1/2},\nonumber\\
    \langle \hat{f},\hat{g}_i\rangle &=|X_i|^{-1}\sum_{n\in X_i\cap\supp f}|f(n)|\nonumber.
\end{align}
Moreover, the functions $Q_i$ satisfy the bounds $|Q_i|\leqslant 1$ and $|1-Q_i|\leqslant |\hat{g}_i|$.

\end{lemma}
\begin{proof}
    That $\lVert \hat{g}_i\rVert_\infty\leqslant 1$ follows immediately from the fact that $|g_i(n)|=|X_i|^{-1}$ if $n\in X_i\cap \supp f$ and $|g_i(n)|=0$ otherwise, and Parseval shows that $\lVert \hat{g}_i\rVert_2\leqslant |X_i|^{-1/2}$. Since $\Phi_1=\hat{g}_1$, the inequality $\lVert \Phi_1\rVert_\infty\leqslant 10$ holds. Assuming now that $\lVert \Phi_j\rVert_\infty\leqslant 10$, then one can observe that $\Phi_{j+1}=\hat{g}_{j+1}+\Phi_jQ_{j+1}$ so that
\begin{align*}
    |\Phi_{j+1}(x)|&\leqslant|\hat{g}_{j+1}(x)|+10e^{-|\hat{g}_{j+1}(x)|}\\
    &\leqslant 10,
\end{align*}
where the final line follows from the basic fact that $y+10e^{-y}\leqslant 10$ whenever $y\in[0,1]$. Parseval's theorem shows that
$$\langle \hat{f},\hat{g}_i\rangle =\sum_{n\in\mathbf{Z}}f(n)\overline{g_i}(n)=|X_i|^{-1}\sum_{n\in X_i\cap \supp f}|f(n)|.$$
Finally, it is trivial that $Q_i=e^{-|\hat{g}_i|}$ is 1-bounded, and the inequality $|1-Q_i|\leqslant |\hat{g}_i|$ follows from the fact that $|e^{-x}-1|\leqslant x$ for $x\geqslant0$.
\end{proof}

We shall also make use of the following important inequality of Rudin \cite{rudin}. To state Rudin's theorem, we need to introduce the notion of dissociativity.
\begin{definition}\label{def:addidimensio}
\normalfont Let $G$ be an Abelian group.  
\begin{itemize}
    \item A set $D \subseteq G$ is \emph{dissociated} if whenever $$\sum_{d\in D}\varepsilon_dd=0$$
for some $\varepsilon_d\in\{-1,0,1\}$, then all $\varepsilon_d=0$. Equivalently, $D$ is dissociated if the set of subset sums $\left\{\sum_{d\in S}d:S\subset D\right\}$ consists of $2^{|D|}$ distinct elements. 
\item The \emph{additive dimension}, which we denote by $\dim (A)$, is the size of the largest dissociated subset of $A$.
\end{itemize}

\end{definition}
\begin{theorem}[Rudin's inequality]\label{th:rudin}
    Let $D\subset \mathbf{Z}$ be a dissociated set and let $f:\mathbf{Z}\to\mathbf{C}$ have $\supp f\subset D$. Then for any $p\in[2,\infty)$ the following bound holds:
    \begin{equation*}
        \lVert \hat{f}\rVert_p\leqslant 10\sqrt{p}\lVert \hat{f}\rVert_2.
    \end{equation*}
\end{theorem}

\section{The Fourier series of $\sum_{a\in A}\left(\phi-1/3\right)(ax)$}
The purpose of this section is to recall some of the Fourier-theoretic setup of Bourgain's paper \cite{bourgain}. We shall assume throughout that $A\subset\mathbf{Z}\setminus\{0\}$ has size $N$. Recall the important basic estimate
\begin{equation}\label{basicestimate}
    S(A)\geqslant \frac{N}{3}+\max_{x}\sum_{a\in A}\left(\phi-\frac{1}{3}\right)(ax),
\end{equation}
where $\phi$ is the characteristic function of $(1/3,2/3)\subset\mathbf{T}$. The function $\phi-1/3:\mathbf{T}\to \mathbf{R}$ has the following Fourier expansion
\begin{equation*}
    \phi(x)-1/3=\frac{-\sqrt{3}}{\pi}\sum_{n=1}^\infty\frac{\chi(n)}{n}c ( nx)
\end{equation*}
where $c(x)=\cos(2\pi x)$ and $\chi$ is the multiplicative character given by
\begin{equation*}
    \chi(n)=\begin{cases}
        0 & \text{if}\ n\equiv 0 \md 3 \\ 
        1 & \text{if}\ n\equiv 1 \md 3 \\
        -1 & \text{if}\ n\equiv 2 \md 3.
    \end{cases}
\end{equation*}
Let $\mu$ denote the M\"obius function, and recall the following fundamental property
\begin{equation*}
    \sum_{k|n}\mu(k)= 1_{n=1}.
\end{equation*} 
For a parameter $Q$, we say that an integer $n$ is $Q$-rough if all its prime factors are greater than $Q$, and we denote the set of $Q$-rough numbers by
\begin{equation*}
    \mathcal{R}_Q=\{n\in \mathbf{N}: (n,p)=1 \text{ for all primes }p\leqslant Q\}.
\end{equation*}
The function appearing in the right hand side of \eqref{basicestimate} is fundamental to our approach, and it will be convenient to introduce notation for the following rescaled version
\begin{align*}
    F(x)=F_A(x)&\vcentcolon=\frac{-\pi}{\sqrt{3}}\sum_{a\in A}(\phi-1/3)(ax)\\
    &=\sum_{a\in A}\sum_{n\geqslant 1}\frac{\chi(n)}{n}c (nax).
\end{align*}
We record some of its important basic properties. For a set $A\subset\mathbf{Z}$, we will write $c_{A}(x)=\sum_{a\in A}c(ax)$ for the cosine polynomial with frequencies in $A$.
\begin{proposition}[Bourgain \cite{bourgain}]\label{prop:basicestimate}
    Let $A\subset\mathbf{Z}\setminus\{0\}$. Then the following holds: 
    \begin{itemize}
        \item [(i)] there exists some absolute constant $c>0$ such that $$S(A)\geqslant \frac{|A|}{3}+c \max_{x\in \mathbf{T}}(-F(x));$$
        \item [(ii)] we have the one-sided estimate $\max_x(-F(x))\geqslant \frac{1}{2}\lVert F\rVert_1$;
        \item [(iii)] for any parameter $Q$ we have \begin{align*}\sum_{k|\prod_{p\leqslant Q}p}\frac{\mu(k)\chi(k)}{k}F(kx)&=\sum_{a\in A}\sum_{n\in\mathcal{R}_Q}\frac{\chi(n)}{n}c (nax)\\
        &= \underbrace{\sum_{a\in A}c( ax)}_{c_A(x)}+\underbrace{\sum_{\substack{1<n\in \mathcal{R}_Q\\ a\in A}}\frac{\chi(n)}{n} c (nax)}_{R_Q(x)},
        \end{align*}
        where the variable $p$ runs over primes only;
        \item [(iv)] for any $Q$ we have $$\lVert F\rVert_1\gg \lVert c_A+R_Q\rVert_1/\log Q.$$
    \end{itemize}
    
\end{proposition}
\begin{proof}
    Item (i) follows immediately from \eqref{basicestimate} with $c=\sqrt{3}/\pi$. For item (ii), we note that $\int_\mathbf{T}(\phi(x)-1/3)\,dx=0$ so that $\int F=0$ and hence $$\int_0^1 |F(x)|\,dx=\int (|F(x)|-F(x))\,dx=\int 2\max(-F(x),0)\,dx\leqslant 2\max_x (-F(x)).$$
    To prove the (iii), we use the multiplicative nature of $\chi$ to calculate
    \begin{align*}
        \sum_{k|\prod_{p\leqslant Q}p}\frac{\mu(k)\chi(k)}{k}\frac{-\pi}{\sqrt{3}}(\phi-1/3)(kx)&=\sum_{k|\prod_{p\leqslant Q}p}\frac{\mu(k)\chi(k)}{k}\sum_{n\geqslant 1}\frac{\chi(n)}{n}c( nkx)\\
        &=\sum_{m\geqslant 1}\frac{\chi(m)}{m}c (mx)\sum_{k|\gcd(m,\prod_{p\leqslant Q}p)}\mu(k)\\
        &= \sum_{m\geqslant 1}\frac{\chi(m)}{m}c (mx) 1_{\{m \text{ is $Q$-rough}\}}\\
        &= c (x) +\sum_{1<m\text{ is $Q$-rough}}\frac{\chi(m)}{m}c( mx).
    \end{align*}
    Replacing $x$ by $ax$ and summing over $a\in A$ yields (iii). For the final item, we use the bound
    \begin{align*}
        \lVert c_A+R_Q\rVert_1&=\left\lVert \sum_{k|\prod_{p\leqslant Q}p}\frac{\mu(k)\chi(k)}{k}F(kx)\right\rVert_1\\
        &\leqslant \sum_{k|\prod_{p\leqslant Q}p}\frac{1}{k}\lVert F\rVert_1\\
        &=\lVert F\rVert_1\prod_{p\leqslant Q}(1+1/p)\\
        &\ll \lVert F\rVert_1(\log Q),
    \end{align*}
    where we used the bound $1+p^{-1}\leqslant e^{p^{-1}}$ and Mertens' estimate \cite[Theorem 2.7]{montgomeryvaughan}: $\sum_{p\leqslant Q}\frac{1}{p}\leqslant \log\log Q+O(1)$.
\end{proof}
Bourgain noted that an important consequence is the following bound for $S(A)$.
\begin{proposition}\label{prop:sumfreeL^1}
    Let $A\subset\mathbf{Z}\setminus\{0\}$ be a set of size $N$. Then
    \begin{equation}
        S(A)\geqslant \frac{N}{3}+c\frac{\lVert c_A\rVert_1}{\log N}
    \end{equation}
    where $c>0$ is some absolute constant. In particular, if $S(A)\leqslant N/3+C$, then $\lVert \hat{1}_A+\hat{1}_{-A}\rVert_1\ll C\log N$.
\end{proposition}
\begin{proof}
   By combining (i),(ii) and (iv) in Proposition \ref{prop:basicestimate} we get $S(A)-N/3\gg (\log Q)^{-1}\lVert c_A+R_Q \rVert_1$, so it suffices to show that $\lVert c_A+R_Q\rVert_1\gg \lVert c_A\rVert_1$ for $Q=100N^2$. Observe that by monotonicity of $L^p$ norms and Parseval's identity $$\lVert R_Q\rVert_1\leqslant \lVert R_Q\rVert_2\leqslant |A|\left\lVert \sum_{1<n\in \mathcal{R}_Q}\frac{\chi(n)}{n}c( nx)\right\rVert_2\leqslant|A|\left(\sum_{n>Q} n^{-2}\right)^{1/2}\leqslant \frac{|A|}{Q^{1/2}}=1/10.$$
   Hence, $\lVert c_A+R_Q\rVert_1\geqslant \lVert c_A\rVert_1-1/10$ and from the trivial lower bound $\lVert c_A\rVert_1\geqslant 1/2$ (which can for example be proved by noting that $\int_0^1 |c_A(x)|\geqslant \int_0^1 c_A(x)e(-ax)=1/2$ for $a\in A$) we see that $\lVert c_A+R_Q\lVert_1\gg \lVert c_A\rVert_1$.
\end{proof}
Propositions \ref{prop:basicestimate} and \ref{prop:sumfreeL^1} can be found in Bourgain's paper, but this is the limit of what his approach yields with regards to Problem \ref{conj:sumfree}. Now that we have introduced the function $F_A$ and its useful relation to $S(A)$ which forms the starting point for our approach, we state the following more detailed theorem.
\begin{theorem}\label{th:mainL^1}
    Let $A\subset\mathbf{Z}\setminus\{0\}$ have size $N$. Then there exists a set $B\subset\mathbf{Z}\setminus\{0\}$ which is $F_4$-isomorphic to $A$ and satisfies $\lVert F_B\rVert_1\gg \log \log N$ where $F_B(x)=\sum_{b\in B}\sum_{n\geqslant 1}\frac{\chi(n)}{n}\cos 2\pi nbx$.
\end{theorem}
For clarity of exposition however, we have written our arguments in this paper to simply give bounds for $S(B)$, where $B$ is $F_4$-isomorphic to $A$, rather than for $\lVert F_B\rVert_1$, but one can check that it is in fact the result above that our proof gives.
Note also that this theorem immediately implies Theorem \ref{th:main}. To see this, recall that $S(A)=S(B)\geqslant N/3+c\lVert F_B\rVert_1$ by (i) and (ii) in Proposition \ref{prop:basicestimate}.
\section{The structure of sets with small $L^1$-norm}
The goal of this section is to prove two structural theorems for sets whose Fourier transform has small $L^1$-norm. The first shows that if $\hat{1}_B$ has small $L^1$ norm for some $B\subset \mathbf{Z}$, then its additive dimension $\dim(B)$ is small. The second shows, again under the assumption that $\lVert \hat{1}_B\rVert_1$ is small, that every large subset of $B$ has large additive energy. In fact, we shall need to prove a more general result which establishes these conclusions under a weaker condition that $\lVert \hat{f} \rVert_1$ is small for a function $f:\mathbf{Z}\to\mathbf{C}$ which satisfies $f(b)\gg 1$ for all $b\in B$. For comparison, note that we always have
$$\log |B|\ll \lVert \hat{1}_B\rVert_1\leqslant |B|^{1/2},$$
where the lower bound follows from the Littlewood $L^1$ conjecture and the upper bound from the simple estimate $\lVert \hat{1}_B\rVert_1\leqslant \lVert \hat{1}_B\rVert_2=|B|^{1/2}$ by Parseval. Both the upper and lower bounds are tight up to a constant factor in general.\footnote{In fact, it is a well-known problem to determine either of these constants.}

\medskip

Before we state the first theorem, the reader may want to recall the Definition \ref{def:addidimensio} of dissociativity and additive dimension.

\begin{theorem}\label{th:smallL^1smalldim}
    Let $f:\mathbf{Z}\to\mathbf{C}$ be a function with $\lVert \hat{f}\rVert_\infty=N$ and let $D\subset \mathbf{Z}$ be any dissociated set. If $\min_{n\in D}|f(n)|\geqslant 1$, then
    \begin{equation}\label{smallL^1smalldimgeneral}
       \lVert \hat{f}\rVert_1\gg \left(\frac{|D|}{\log N}\right)^{1/2}. 
    \end{equation}
\end{theorem}
\begin{cor}\label{cor:smalldim}
    Let $B\subset\mathbf{Z}$ be a finite set of integers. Then
    \begin{equation}\label{smallL^1smalldim}
       \dim(B)\ll \lVert \hat{1}_B\rVert_1^{2} \log |B|. 
    \end{equation}
\end{cor}
\begin{rmk*}
   The bound in \eqref{smallL^1smalldimgeneral} is best possible up to a constant factor for general $f$. As an example, one can take a Fej\'er kernel $\hat{f}=F_N(x)=\sum_{n=-N}^N\left(1-\frac{|n|}{N}\right)e(nx)$ which has $\lVert F_N\rVert_1=1$, while $f(n)\geqslant 1/2$ for all $n$ in the dissociated set $\{2^j:j<\log_2 N-1\}$.
\end{rmk*}

\begin{proof}[Proof of Theorem \ref{th:smallL^1smalldim}]
Let $D\subset\mathbf{Z}$ be dissociated. We define the function $g:D\to\mathbf{C}$ by $g(d)=|D|^{-1}\frac{f(d)}{|f(d)|}$. We also define the corresponding `correction' function $Q(t)=\exp(-|\hat{g}(t)|)$ and note that $|Q(t)|\leqslant 1$ and $|1-Q(t)|\leqslant |\hat{g}(t)|$, by Lemma \ref{lem:basicMPS}. We further know from Lemma \ref{lem:basicMPS} that if $\Phi_j$ is defined as follows:
\begin{align*}
    \Phi_{j}(t)&=\hat{g}(t)(1+Q(t)+\dots+Q^{j-1}(t)),
\end{align*}
then $\lVert \Phi_j\rVert_\infty\leqslant 10$ for all $j$.

\medskip

Let us fix an integer $J$. By a telescoping identity, we see that \begin{align*}\Phi_J(t)&=J\hat{g}(t)-\sum_{j=1}^{J-1}\sum_{k=1}^j\hat{g}(t)(1-Q(t))Q(t)^{k-1}.
\end{align*}Then as $\lVert \Phi_J\rVert_\infty\leqslant 10$, we have
\begin{align}
    \lVert \hat{f}\rVert_1&\gg \langle \hat{f},\Phi_J\rangle\nonumber\\
    &= J\langle \hat{f},\hat{g}\rangle-E\nonumber\\
    &\geqslant J\min_{d\in D}|f(d)|-|E|,\label{eq:JEbound}
\end{align}
where we used the last equation in \eqref{basicMPS} and we defined $$E=\sum_{1\leqslant k\leqslant j\leqslant J-1}\langle \hat{f},\hat{g}(1-Q)Q^{k-1}\rangle.$$ To conclude, we bound $E$ and optimise the choice of $J$. Note that by the properties of $\hat{g}$ and $Q$ we get
\begin{align}
    |E|&\leqslant \sum_{1\leqslant k\leqslant j\leqslant J-1}\langle |\hat{f}|,|\hat{g}|^2\rangle\nonumber\\
    &\leqslant J^2 \lVert \hat{f}\rVert_p\lVert \hat{g}\rVert_{2q}^2,\label{eq:Eupp}
\end{align}
where we used H\"older's inequality with exponent pair $(p,q)$ such that $\frac{1}{p}+\frac{1}{q}=1$. To estimate the first $L^p$-norm in terms of the $L^1$-norm of $\hat{f}$, we take $p=1+1/\log N$ so that $|\hat{f}|^p\ll |\hat{f}|$ because of our assumption that $\lVert \hat{f}\rVert_\infty =N$, and hence
\begin{align*}
    \lVert \hat{f}\rVert_p\ll \lVert \hat{f}\rVert_1^{1/p}\leqslant \lVert \hat{f}\lVert_1
\end{align*}
using also that $\lVert \hat{f}\rVert_1 \geqslant \langle \hat{f}, e(d\cdot)\rangle=1$ for $d\in D$. To bound $\lVert \hat{g}\rVert_{2q}$, we use the fact that $\supp g\subseteq D$ is dissociated so that by Rudin's inequality in Theorem \ref{th:rudin} we obtain
\begin{align*}
    \lVert \hat{g}\rVert_{2q}\ll q^{1/2}\lVert \hat{g}\rVert_2\ll q^{1/2}|D|^{-1/2},
\end{align*}
where we used Parseval to evaluate $\lVert \hat{g}\rVert_2$.
In total, since $q\ll \log N$ we can bound \eqref{eq:Eupp} by
\begin{align*}
    |E|&\ll J^2\lVert \hat{f}\rVert_1 \frac{\log N}{|D|},
\end{align*}
and we can substitute this back in \eqref{eq:JEbound} to deduce that
\begin{align*}
    \lVert \hat{f}\rVert_1+J^2\lVert \hat{f}\rVert_1\frac{\log N }{|D|}\gg J\min_{n\in D}|f(n)|\geqslant J.
\end{align*}
Taking $J=\lfloor(|D|/\log N)^{1/2}\rfloor$ shows that $\lVert \hat{f}\rVert_1\gg (|D|/\log N)^{1/2}$ as desired.

\end{proof}
\begin{rmk*}
    \normalfont The author would like to thank Thomas Bloom for pointing out that Zygmund \cite[Chapter XII, (7.6)]{zygmund}  proves that if $(n_j)$ is a lacunary (i.e. $n_{j+1}/n_j>c>1$) and $\hat{f}(\log^+|\hat{f}|)^{1/2}$ is integrable, then $\sum_j |f(n_j)|^2$ converges. Interestingly, Pichorides noted that this proof works when $\{n_j\}$ is dissociated and that a quantitative version of his argument yields the bound in Theorem \ref{th:smallL^1smalldim}. Pichorides \cite{pichorides} in fact used this to establish what was at the time the best bound towards the Littlewood $L^1$ conjecture, and Konyagin \cite{konyaginl} makes use of similar ideas (but about the number of distinct $2$-adic valuations of the $n_j$ rather than dissociativity). We include the short proof above since it is different and in particular constructs an explicit test function witnessing that $\lVert \hat{f}\rVert_1$ is large.
\end{rmk*}
In the Proposition \ref{prop:sumfreeL^1}, we showed that if $A$ is a set of $N$ integers such that $S(A) \leqslant N/3+C$, then $\lVert \hat{1}_A+\hat{1}_{-A}\rVert \ll C\log N$. Applying Theorem \ref{th:smallL^1smalldim} with $f=1_A+1_{-A}$ therefore shows the following.
\begin{cor}\label{cor:sum-freesmalldim}
    Let $A\subset\mathbf{Z}\setminus\{0\}$ have size $N$ and let $S(A)\leqslant N/3+C$. Then $\dim(A)\ll C^2(\log N)^3$.
\end{cor}
The fact that $A$ has small dimension implies that $A$ has strong additive structure as we will prove in the next section. 

\medskip

We first discuss the next theorem, which roughly speaking states that if $\lVert \hat{f}\rVert_1$ is small, then every large set $A$ with $\min_{a\in A}|f(a)|\gg 1$ has large additive energy. For sets $B,B'\subset\mathbf{Z}$, we define the joint additive energy $$E(B,B')=\#\{(b_1,b_2\in B, b_1',b_2'\in B':b_1-b_2=b_1'-b_2'\}$$ and the additive energy of a set is defined to be $E(B)=E(B,B)$. We also point out that this theorem is used to obtain the structure Theorem \ref{th:mainstructure}, but that it is not required for Theorem \ref{th:main}. 
\begin{theorem}\label{th:smallL^1largeenergy}
    Let $f:\mathbf{Z}\to\mathbf{C}$ be a function with $\lVert \hat{f}\rVert_2\ll N^{1/2}$ and let $K=100\lVert \hat{f}\rVert_1$. Let $X_1,X_2,\dots,X_K\subset \mathbf{Z}$ be any sets such that $\min_{n\in X_i}|f(n)|\geqslant 1/2$. Then there exists distinct $j,j'\in[K]$ such that
    \begin{equation}\label{largeenergythm}
       E(X_j,X_{j'})\gg K^{-2}\frac{|X_j|^2|X_{j'}|^2}{N}. 
    \end{equation}
\end{theorem}

\begin{proof}
     We argue by contradiction, assuming that we can find $X_1,X_2,\dots,X_K$ such that $\min_{n\in X_i}|f(n)|\geqslant 1/2$ and
    \begin{equation*}
       E(X_j,X_{j'})\leqslant cK^{-2}\frac{|X_j|^2|X_{j'}|^2}{N} 
    \end{equation*} for all $j<j'$, where $c>0$ is some absolute constant to be determined later. Define the functions
    \begin{align*}
        g_i:\mathbf{Z}\to\mathbf{C}:g_i(n)=\begin{cases}
            |X_i|^{-1}\frac{f(n)}{|f(n)|} & \text{if $n\in X_i$,}\\
            0 & \text{otherwise},
        \end{cases}
    \end{align*}
and further define $Q_i(x)=e^{-|\hat{g}_i(x)|}$ and
\begin{equation*}
    \Phi_j=\hat{g}_j+\hat{g}_{j-1}Q_j+\dots+\hat{g}_1Q_2\dots Q_j.
\end{equation*}
We again have the basic inequalities $|Q_j(t)|\leqslant 1, |1-Q_j(t)|\leqslant |\hat{g}_j(t)|$ and note that $\lVert \Phi_j\rVert_\infty\leqslant 10$ for all $j$ by Lemma \ref{lem:basicMPS}. We define
$$Z(t)=\sum_{j=1}^K \hat{g}_j(1-Q_{j+1}\dots Q_K)$$ and hence,
\begin{align}
    10\lVert \hat{f}\rVert_1&\geqslant \langle \hat{f},\Phi_K\rangle\nonumber\\
    &= \sum_{j=1}^K\langle \hat{f},\hat{g}_j\rangle-\langle \hat{f},Z\rangle\nonumber\\
    &\geqslant K\min_{n\in \cup_j X_j}|f(n)|-|\langle \hat{f},Z\rangle|,\label{J,Ebound}
\end{align}
where we used the last equation in \eqref{basicMPS}. By a telescoping identity, we may rewrite \begin{align*}Z(t)=\sum_{j=1}^{K-1}\sum_{k=j+1}^K\hat{g}_j(t)(1-Q_k(t))Q_{j+1}(t)\dots Q_{k-1}(t)
\end{align*}
which yields the following upper bound
\begin{align*}
    |\langle \hat{f},Z\rangle|&\leqslant \sum_{1\leqslant j<k\leqslant K}\langle |\hat{f}|,|\hat{g}_j||\hat{g}_k|\rangle\\
    &\leqslant \sum_{1\leqslant j<k\leqslant K} \lVert \hat{f}\rVert_2\lVert\hat{g}_j\hat{g}_k\rVert_2
\end{align*}
by the Cauchy-Schwarz inequality. 
By assumption we can bound $\lVert \hat{f}\rVert_2 \ll N^{1/2}$. As $\hat{g}_i(t)=|X_i|^{-1}\sum_{n\in X_i}\frac{f(n)}{|f(n)|}e(nt)$, the norms $\lVert\hat{g}_j\hat{g}_k\rVert_2$ can be explicitly calculated using the orthogonality of characters as follows: 
\begin{align*}
|X_j|^2|X_k|^2\lVert\hat{g}_j\hat{g}_k\rVert_2^2&=\int_0^1\left|\sum_{n\in X_j}\frac{f(n)}{|f(n)|}e(nt)\right|^2\left|\sum_{n\in X_k}\frac{f(n)}{|f(n)|}e(nt)\right|^2\,dt\\
    &=\sum_{\substack{n_1,n_2\in X_j,n_3,n_4\in X_k\\n_1-n_2=n_3-n_4}}\frac{f(n_1)\overline{f(n_2)f(n_3)}f(n_4)}{|f(n_1)f(n_2)f(n_3)f(n_4)|}\\
    &\leqslant \#\{n_1,n_2\in X_j,n_3,n_4\in X_k:n_1-n_2=n_3-n_4\}\\
    &=E(X_j,X_k).
\end{align*}
As we are assuming that $E(X_j,X_k)\leqslant cK^{-2}|X_j|^2|X_k|^2/N$, we conclude that 
\begin{align}\label{largeenergyEbound}
    |\langle \hat{f},Z\rangle|&\ll \sum_{1\leqslant j<k\leqslant K} N^{1/2}c^{1/2}K^{-1}N^{-1/2}\leqslant c^{1/2}K.
\end{align}
Recall that $K=100\lVert \hat{f}\rVert_1$ and that $\min_{n\in \cup_j X_j}|f(n)|\geqslant 1/2$ so combining \eqref{J,Ebound} and \eqref{largeenergyEbound} produces the inequality
\begin{align*}
    \frac{K}{10}\geqslant 10\lVert \hat{f}\rVert_1\geqslant K \min_{n\in \cup_j X_j}|f(n)|-O(c^{1/2}K)\geqslant K/2-O(c^{1/2}K).
\end{align*}
This gives a contradiction upon choosing $c>0$ to be sufficiently small.

\end{proof}
One can apply Theorem \ref{th:smallL^1largeenergy} with $X_1=\dots=X_K=X$ to deduce the following result.
\begin{cor}\label{cor:smallL^1largeenergy}
    Let $f:\mathbf{Z}\to\mathbf{C}$ be a function with $\lVert \hat{f}\rVert_2\leqslant N^{1/2}$. Let $X\subset \mathbf{Z}$ be any set such that $\min_{n\in X}|f(n)|\geqslant 1/2$. Then 
    \begin{equation}\label{largeenergythm}
       E(X)\gg \lVert \hat{f}\rVert_1^{-2}\frac{|X|^4}{N}. 
    \end{equation}
\end{cor}

\section{Sets with small dimension have a `dense' model}
In this section we will show that sets of integers with small additive dimension are Freiman isomorphic to sets which have relatively large density on an interval. For example, we will show that a set $B\subset \mathbf{Z}$ with $\dim(B)\leqslant (\log |B|)^{O(1)}$ has a Freiman isomorphic copy $B'$ with $B'\subset[-e^{(\log |B|)^{O(1)}},e^{(\log |B|)^{O(1)}}]$. For reference, it is a well-known fact in additive combinatorics that any set $B\subset \mathbf{Z}$ has a Freiman isomorphic copy $B'\subset[-e^{O(|B|)},e^{O(|B|)}]$; the point of this section is to show that one can obtain significantly stronger bounds for sets with small dimension. The results in this section hold for any reasonable notion of Freiman isomorphism.
\begin{definition}
Let $G,G'$ be Abelian groups and let $A\subset G$, $A'\subset G'$. We say that a map $\phi:A\to A'$ is an \textit{$F_\ell$-homomorphism} if whenever $a_1,a_2,\dots,a_\ell\in A$ satisfy $$\varepsilon_1a_1+\varepsilon_2a_2+\dots+\varepsilon_\ell a_\ell=0$$
for some $\varepsilon_j\in\{-1,0,1\}$,
then $$\varepsilon_1\phi(a_1)+\varepsilon_2\phi(a_2)+\dots+\varepsilon_\ell\phi(a_\ell)=0.$$ We say that $A$ and $A'$ are \textit{$F_\ell$-isomorphic} if there is a bijective $F_\ell$-homomorphism $\phi:A\to A'$ so that $\phi^{-1}$ is also an $F_\ell$-homomorphism.
\label{frei}
\end{definition}

Our goal is to prove the following theorem, showing that sets with small additive dimension have a dense Freiman model. We note that Green and Ruzsa \cite{greenruzsa} proved a similar result under the assumption that the set has small doubling instead.
\begin{theorem}\label{th:densemodel}
    Let $A\subset \mathbf{Z}$ be a set such that any $F_\ell$-isomorphic set $A'\subset \mathbf{Z}$ satisfies $\dim(A')\leqslant k$.
    Then $A$ is $F_\ell$-isomorphic to a subset $B\subset [-T,T]$ where $T\ll \ell^{O(k\log k)}$.
 
\end{theorem}
The importance of this result in our setting is as follows. 
\begin{cor}[`Dense' model I]\label{cor:densemodel}
    Let $A\subset \mathbf{Z}$ be a set of size $N$ with $S(A)\leqslant N/3+C$. Then $A$ is $F_4$-isomorphic to a set $B\subset [-T,T]$ where $T\leqslant e^{O((C\log N)^4)}$.
\end{cor}
\begin{proof}[Proof of Corollary \ref{cor:densemodel}]
Let $A'\subset \mathbf{Z}$ be a set which is $F_4$-isomorphic to $A$. By Corollary \ref{cor:sum-freesmalldim} we have that $\dim(A')\ll C^2(\log N)^3$. The result now follows from Theorem \ref{th:densemodel} (with a somewhat stronger bound than we claimed here).
\end{proof}
We shall not be concerned with optimising $T$ for now; $T=e^{O((C\log N)^4)}$ is more than sufficient for our proof of Theorem \ref{th:main}. In section 9, we will show that if $S(A)\leqslant N/3+C$ then $A$ has a significantly denser $F_4$-model inside $[-T,T]$, where $T\leqslant N^{C^{O(1)}}$.
To prove Theorem \ref{th:densemodel}, we combine several combinatorial lemmas.
\begin{lemma}\label{lem:dissociated-set-generator}
Let $B \subseteq G$ be a finite subset of an Abelian group.  If $D$ is a maximal dissociated subset of $B$, then
$$B \subseteq \Span(D) \vcentcolon= \left\{\sum_{d\in D}\varepsilon_d d: \varepsilon_d\in\{-1,0,1\}\right\}.$$
\end{lemma}

\begin{proof}
By maximality, for every element $b\in B\setminus D$, the set $\{b\}\cup D$ is not dissociated; rearranging a relation in $\{b\}\cup D$ with $\pm 1$-coefficients shows that $b$ lies in $\Span(D)$.
\end{proof}

We now show that, in cyclic groups $\mathbf{Z}/p\mathbf{Z}$ of prime order, sets with small additive dimension can be dilated to lie in a short interval around $0$. Here, we denote the dilated set by $\lambda\cdot B=\{\lambda b:b\in B\}$. The use of such so-called rectification arguments in additive combinatorics goes back at least to \cite{bilulevruzsa}.

\begin{lemma}\label{lem:small-dim-rectification}
If $B \subseteq \mathbf{Z}/p\mathbf{Z}$ is a subset of dimension $\dim(B)\leqslant k$, then there is some $\lambda \in (\mathbf{Z}/p\mathbf{Z})^\times$ such that the dilate $\lambda\cdot B$ is contained in the interval $[-kp^{1-1/k},kp^{1-1/k}]$.
\end{lemma}

\begin{proof}
Let $D$ be a maximal dissociated subset of $B$, so that $|D|\leqslant k$. Consider the set
$$\{(\lambda d/p)_{d\in D}: \lambda \in \mathbf{Z}/p\mathbf{Z}\} \subseteq \mathbf{T}^{k}.$$
We divide $\mathbf{T}^{k}$ into boxes of the form $\prod_{i=1}^k[j_i/p^{1/k},(j_i+1)/p^{1/k})$ where the $j_i$ range over the integers in $[0,p^{1/k})$. The pigeonhole principle provides distinct $\lambda_1, \lambda_2 \in \mathbf{Z}/p\mathbf{Z}$ such that $\|\lambda_1 d/p-\lambda_2 d/p\|_{\mathbf{T}} \leqslant p^{-1/k}$ for all $d\in D$.  Take $\lambda\vcentcolon=\lambda_1-\lambda_2 \in (\mathbf{Z}/p\mathbf{Z})^\times$, so that $\lambda d \in [-p^{1-1/k},p^{1-1/k}]$ for all $d\in D$.  Since $B\subseteq \operatorname{span}(D)$ by Lemma~\ref{lem:dissociated-set-generator}, we have that
$\lambda \cdot B \subseteq [-kp^{1-1/k},kp^{1-1/k}].$
\end{proof}
The next lemma provides a procedure for finding $F_\ell$-isomorphic copies of a set.
\begin{lemma}\label{lem:dilating}
    Let $B\subset\mathbf{Z}$ satisfy $B\subset(-p/\ell,p/\ell)$ where $p$ is a prime. Denote by $\pi:(-p/2,p/2)\to\mathbf{Z}/p\mathbf{Z}$ the projection map. If $\lambda\in(\mathbf{Z}/p\mathbf{Z})^\times$ has that $\lambda\cdot \pi(B)\subset(-p/\ell,p/\ell)\subset\mathbf{Z}/p\mathbf{Z}$, then $B$ is $F_\ell$-isomorphic to $\pi^{-1}(\lambda\cdot\pi(B))\subset(-p/\ell,p/\ell)\subset\mathbf{Z}$.
\end{lemma}
\begin{proof}
    $B$ is $F_\ell$-isomorphic to $\pi(B)$ precisely because the conditions $\sum_{j=1}^\ell \varepsilon_jx_j=0$ and $\sum_{j=1}^\ell \varepsilon_jx_j\equiv 0 \pmod p$ are equivalent for integers $x_j\in (-p/\ell,p/\ell)$ and $\varepsilon_j\in\{-1,0,1\}$. Furthermore, it is trivial that $C$ and any dilate $\lambda\cdot C$ are $F_\ell$-isomorphic for any $C\subset\mathbf{Z}/p\mathbf{Z}$ and $\lambda\in(\mathbf{Z}/p\mathbf{Z})^\times$.
\end{proof}
\begin{proof}[Proof of Theorem \ref{th:densemodel}]
    Suppose that $A\subset \mathbf{N}$ is a set such that any $F_\ell$-isomorphic set $A'\subset \mathbf{Z}$ satisfies $\dim(A')\leqslant k$. We further suppose that $A'\subset \mathbf{Z}$ has that $m=m(A')\vcentcolon=\max_{x\in A'}|x|$ is minimal over all sets $A'$ which are $F_\ell$-isomorphic to $A$. We may find a prime $p\in (\ell m,2\ell m]$ and by Lemma \ref{lem:dilating}, $A'$ is $F_\ell$-isomorphic to $\pi(A')\subset\mathbf{Z}/p\mathbf{Z}$. It is trivial that $\dim(\pi(A'))\leqslant\dim(A')\leqslant k$ (note that the first notion of additive dimension is taken in $\mathbf{Z}/p\mathbf{Z}$ while the second is in $\mathbf{Z}$). If it were the case that $kp^{1-1/k}<m$, then $kp^{1-1/k}<p/\ell$ so by combining Lemmas \ref{lem:small-dim-rectification} and \ref{lem:dilating} we would obtain a set $A''$ which is $F_\ell$-isomorphic to $A$ and moreover has
    \begin{align*}
        m(A'')\leqslant kp^{1-1/k}<m.
    \end{align*}
    These properties contradict our choice of $A'$ so we must have that
    $k(2\ell m)^{1-1/k}\geqslant kp^{1-1/k}\geqslant m$. We deduce that $m\ll (2\ell k)^k$ as desired.
\end{proof}
In the rest of this section, we study the structure of $F_A$ for sets $A$ with $S(A)\leqslant N/3+C$ which are contained in some `short' interval. By Corollary \ref{cor:densemodel}, we may indeed consider sets $A\subset[-T,T]$ where $T\leqslant e^{(\log N)^5}$ from now on, as all sets $A'$ which do not have such a `dense' $F_4$-model must satisfy $S(A')\geqslant N/3+c(\log N)^{1/4}$. We will use the following lemma which one can think of as providing good bounds on the size of the Fourier coefficients of $R_Q$, which we defined in Proposition \ref{prop:basicestimate} as
\begin{align*}
        R_Q(x)=\sum_{\substack{1<n\in \mathcal{R}_Q\\ a\in A}}\frac{\chi(n)}{n}c( nax).
    \end{align*}
We need to prove such a result in a somewhat more general setting.
\begin{lemma}\label{lem:coefbound1}
    Let $B_s\subset \mathbf{Z}\setminus\{0\}$, $s\in\mathcal{S}$, be a collection of disjoint sets. Let $Q>1$ and $T\leqslant e^{Q}$ be parameters and let $k$ be a non-zero integer. Define for each $s\in\mathcal{S}$ the function
    $$R_{B_s}(x)=\sum_{b\in B_s}\sum_{1<n\in\mathcal{R}_Q}\frac{\alpha(n,s)}{n}e(nkbx)$$
    where the $\alpha(n,s)$ are $1$-bounded complex numbers, and where $\mathcal{R}_Q$ is the set of $Q$-rough numbers.
    Then
    $R(x)\vcentcolon=\sum_sR_{B_s}(x)$ satisfies the bound $|\hat{R}(m)|\ll \frac{\log T}{Q}$ for all its Fourier coefficients with frequencies $m\in[-10T,10T]$.
\end{lemma}
\begin{proof}
This can be proved as follows: \begin{align*}
    |\hat{R}(m)|&\leqslant \sum_s\sum_{b\in B_s}\sum_{\substack{1<n\in\mathcal{R}_Q\\\exists b\in B_s \text{ with }nkb=  m}}\frac{1}{n}\\
    &\leqslant \sum_{\substack{1<n\in\mathcal{R}_Q\\ n| m}}\frac{1}{n}
\end{align*}
which one can check by noting that each divisor $n\in\mathcal{R}_Q$ of $m$ can contribute at most once since the sets $B_s$ are disjoint and the relation $nkb=m$ uniquely determines $b$ (if it exists). As $\mathcal{R}_Q$ is the set of $Q$-rough numbers, we can bound this by
\begin{align*}
    |\hat{R}(m)|&\ll -1+\prod_{Q<p \text{ prime divisor of $|m|$}}\left(1+\frac{1}{p}+\frac{1}{p^2}+\dots\right)\\
    &\ll -1+\left(1+2Q^{-1}\right)^{\omega(|m|)},
\end{align*}
where $\omega$ counts the number of distinct prime divisors. Using the trivial bound $\omega(|m|)\ll \log T\leqslant Q$ for all $m\in[-10T,10T]$ and the basic inequality $(1+x)^y-1\leqslant e^{xy}-1\ll xy$ which is valid for $x>0$ and all $0\leqslant y \ll 1/x$, we obtain the claimed bound $|\hat{R}(m)|\ll \omega(|m|)Q^{-1}\ll (\log T)Q^{-1}$.
\end{proof}

We show that such a result can be used to deduce strong $L^2$-bounds for functions of the type discussed below. For technical reasons that will become clear later, we will in practice bound the $L^2$-norm of a certain truncation of their Fourier series, and to obtain such truncations we recall the definition of de la Vall\'ee-Poussin kernels.
    \begin{definition}
        \normalfont We define the \emph{de la Vall\'ee-Poussin kernel}
$$V_{T}(x)=\sum_{n=-2T}^{-T-1}\left(1-\frac{|2T+n|}{T}\right)e(nx)+\sum_{n=-T}^Te(nx)+\sum_{n=T+1}^{2T}\left(1-\frac{|2T-n|}{T}\right)e(nx).$$ 
    \end{definition}
The de la Vall\'ee-Poussin kernels clearly have the important basic property that $\hat{V}_{T}(n)=1$ for $|n|\leqslant T$ while $\hat{V}_{T}(n)=0$ for $|n|\geqslant 2T$.
The following rather technical looking lemma provides properties that will be crucial in two later stages of this paper. We also emphasise that the variable $p$ runs over primes only.
\begin{lemma}\label{lem:generalcont}
    Let $1\leqslant Q_1<Q$ and $T\leqslant e^{Q}$ be parameters, let $(\nu_p)_{p\leqslant Q_1}\in\mathbf{N}^{\pi(Q_1)}$ and let $r\in \prod_{p\leqslant Q_1}(\mathbf{Z}/p\mathbf{Z})^\times$. Let $B_s\subset \mathbf{Z}\setminus\{0\}$, $s\in\prod_{p\leqslant Q_1}(\mathbf{Z}/p\mathbf{Z})^\times$, be a collection of sets of size at most $K$. Assume that $k$ is a non-zero integer such that for each $s$ and each $b\in B_s$ the congruence $bk\equiv s\prod_{p\leqslant Q_1}p^{\nu_p}\mdlem{\prod_{p\leqslant Q_1}p^{\nu_p+1}}$ holds. Define
    $$B_s(x)=\sum_{b\in B_s}\sum_{\substack{n\in\mathcal{R}_Q\\ n\equiv rs^{-1}\mdsublem{\prod_{p\leqslant Q_1} p}}}\frac{\chi(n)}{n}e(nkbx).$$
    Then
    $$\sum_sB_s(x)=\sum_{b\in B_r}e(bkx)+E(x)$$
        for some function $E:\mathbf{T}\to\mathbf{C}$ which satisfies $\lVert E*V_{T}\rVert_2\ll \frac{\log T}{Q^{1/2}}K^{1/2}$.
\end{lemma}
\begin{proof}
Recall that $\mathcal{R}_Q$ is the set of $Q$-rough numbers. Trivially, the only value of $s\in\prod_{p\leqslant Q_1}(\mathbf{Z}/p\mathbf{Z})^\times$ such that $n=1$ satisfies $n\equiv rs^{-1}\md{\prod_{p\leqslant Q_1}p}$ is $s=r$. Hence, if we write
\begin{equation}\label{eq:Edefi}
    E(x)=\sum_s\sum_{b\in B_s}\sum_{\substack{1<n\in\mathcal{R}_Q \\ n\equiv rs^{-1}\mdsub{\prod p}}} \frac{\chi(n)}{n}e(nbkx),
\end{equation}
then
$\sum_s B_s(x)=\sum_{b\in B_r}e(bkx)++ E(x)$.
It remains to show that $\lVert E*V_{T}\rVert_2\ll (\log T)Q^{-1/2}K^{1/2}$. Note that $E(x)$ is precisely of the form that we studied in the previous lemma: the congruence condition implies that the sets $B_s$ are pairwise disjoint and we take $\alpha(n,s)=1_{n\equiv rs^{-1}\mdsub{\prod p}}\chi(n)$. Hence, $\max_{|m|\leqslant 10T}|\hat{E}(m)|\ll \frac{\log T}{Q}$. Observe that $E*V_{T}$ is a trigonometric polynomial of degree at most $2T$ as $(E*V_T)^\wedge(n)=\hat{E}(n)\hat{V}_{T}(n)$, and note that $|\hat{V}_{T}(n)|\leqslant 1$, so by Parseval we can bound
\begin{align*}
    \lVert E*V_{T}\rVert_2^2&=\sum_{n=-2T}^{2T}|\hat{E}(n)|^2|\hat{V}_{T}(n)|^2\leqslant \sum_{n=-2T}^{2T}|\hat{E}(n)|^2\\
    &\leqslant \max_{|n|\leqslant 2T}|\hat{E}(n)|\sum_{n=-2T}^{2T}|\hat{E}(n)|\ll \frac{\log T}{Q}\sum_{n=-2T}^{2T}|\hat{E}(n)|,
\end{align*}
so it suffices to show that $\sum_{|n|\leqslant 2T}|\hat{E}(n)|\ll K\log T$. This follows from the explicit form \eqref{eq:Edefi} of $E$ as $\chi$ is $1$-bounded: 
\begin{align*}
    \sum_{m=-2T}^{2T}|\hat{E}(m)|&\leqslant \sum_{s}\sum_{b\in B_s}\sum_{\substack{1<n\in\mathcal{R}_Q\\ n\equiv rs^{-1}\mdsub{\prod p}}}\frac{1_{|nkb|\leqslant 2T}}{n}\\
    &\leqslant \sum_s |B_s|\sum_{\substack{
    n\in\mathcal{R}_Q\\\ n\equiv rs^{-1}\mdsub{\prod p}}}\frac{1_{n\leqslant 2T}}{n}\\
    &\leqslant K\sum_{n\leqslant 2T}\frac{1}{n}\ll K\log T,
\end{align*}
where we used that $\max_s|B_s|\leqslant K$.
\end{proof}

\section{The distribution of $A$ modulo small primes
}We study the distribution of the set $A$ in residue classes modulo powers of small primes under the assumption that $S(A)\leqslant N/3+C$. Let $Q_1$ be a parameter. We define for $r\in \prod_{p\leqslant Q_1}\left(\mathbf{Z}/p\mathbf{Z}\right)^\times$ and $(\nu_p)_{p\leqslant Q_1}\in\mathbf{N}^{\pi(Q_1)}$ the following subset of $A$:
\begin{equation*}
    A(r,(\nu_p))=A\cap\left\{n\in\mathbf{Z}: n\equiv r\prod_{p\leqslant Q_1}p^{\nu_p} \md {\prod_{p\leqslant Q_1}p^{\nu_p+1}}\right\}.
\end{equation*}
We emphasise that the variable $p$ always runs over primes only and that, for our purposes, $\mathbf{N}$ contains 0.
We first show, using the assumption that $\dim(A)$ is small, that a large portion of $A$ must be contained in a single such set $A(r,(\nu_p))$.
\begin{lemma}\label{lem:largeresidueclass}
    Let $A\subset\mathbf{Z}$ be a set of $N$ integers. Then there exist $r\in \prod_{p\leqslant Q_1}\left(\mathbf{Z}/p\mathbf{Z}\right)^\times$ and $(\nu_p)_{p\leqslant Q_1}\in\mathbf{N}^{\pi(Q_1)}$ such that $$|A(r,(\nu_p))|\geqslant \frac{N}{(\dim(A))^{\pi(Q_1)}\prod_{p\leqslant Q_1}p}.$$
\end{lemma}
\begin{proof}
    Let $p$ be a prime, and for an integer $n$ let $\nu_p(n)$ be the exact power of $p$ dividing $n$. We claim that $\{\nu_p(a):a\in A\}$ has size at most $\dim(A)$. Indeed, if $\nu_p(a_1)<\nu_p(a_2)<\dots<\nu_p(a_k)$, then considering a possible relation $\sum_{j=1}^k\varepsilon_ja_j=0$ with $\varepsilon_j\in\{-1,0,1\}$ modulo $p^{\nu_p(a_i)+1}$ for all $i$ shows that each $\varepsilon_i=0$, thus implying that $\{a_1,a_2,\dots,a_k\}$ is dissociated. Hence, the image of the map
    $$\rho:A\to\mathbf{N}^{\pi(Q_1)}:a\mapsto(\nu_p(a))_{p\leqslant Q_1}$$
    has size at most $(\dim(A))^{\pi(Q_1)}$ and it follows that there exists a $(\nu_p)_{p\leqslant Q_1}$ such that its preimage $$\rho^{-1}((\nu_p))=\bigcup_{r\in \prod_{p\leqslant Q_1}\left(\mathbf{Z}/p\mathbf{Z}\right)^\times}A(r,(\nu_p))$$
    has size at least $N(\dim(A))^{-\pi(Q_1)}$. Clearly, there then exists an $r\in \prod_{p\leqslant Q_1}\left(\mathbf{Z}/p\mathbf{Z}\right)^\times$ satisfying the conclusion of the lemma.
\end{proof}
This implies the following structure for sets $A$ for which $S(A)$ is small. The exact choice of the parameter $Q_1$ is not important here, or in the rest of the paper (any choice $Q_1=(\log N)^c$ for $c\in(0,1)$ will do). For clarity, we shall from now on always take $Q_1=(\log N)^{1/2}$.
\begin{cor}\label{cor:largeA(s,mu)}
    Let $Q_1=(\log N)^{1/2}$. Let $A\subset\mathbf{Z}\setminus\{0\}$ be a set of size $N$. Then either $S(A)\geqslant N/3+c(\log N)^{1/2}$ or else the following holds. There exist $r\in \prod_{p\leqslant Q_1}\left(\mathbf{Z}/p\mathbf{Z}\right)^\times$ and $(\nu_p)_{p\leqslant Q_1}\in\mathbf{N}^{\pi(Q_1)}$ such that $|A(r,(\nu_p))|\geqslant_\varepsilon N^{1-\varepsilon}$.
\end{cor}
\begin{proof}
    Corollary \ref{cor:sum-freesmalldim} implies that either $S(A)\geqslant N/3+c(\log N)^{1/2}$ so that we are done, or else that $\dim(A)\ll (\log N)^4$. One can now simply use Lemma \ref{lem:largeresidueclass} and calculate that $(\dim(A))^{\pi(Q_1)}\prod_{p\leqslant Q_1}p\leqslant e^{O((\log N)^{1/2})}\ll_\varepsilon N^\varepsilon$ when $Q_1=(\log N)^{1/2}$.
\end{proof}
The main result of this section is Proposition \ref{prop:distmodpstruct} which shows that the collection of sizes of the sets $A(r,(\nu_p))$ exhibits strong structure when $S(A)$ is small. The statement and proof of Proposition \ref{prop:distmodpstruct} are somewhat technical so we discuss the following model setting first. Its proof contains many of the main ideas but allows us to ignore error term contributions that show up in the general case.
\begin{proposition}[Model setting]\label{prop:distmodel1}
    Let $A\subset\mathbf{Z}\setminus\{0\}$ and assume that $A\subset [-T,T]$ where $T\ll e^{O((\log N)^5)}$. Let $Q_1=(\log N)^{1/2}$ and suppose that $$\max_{r\in\prod_{p\leqslant Q_1}(\mathbf{Z}/p\mathbf{Z})^\times}|A(r,(0))|\geqslant (\log N)^4,$$ where $(0)=(0)_{p\leqslant Q_1}$. Then $\lVert F_A\rVert_1\gg \log\log N$, and in particular $S(A)\geqslant N/3+c\log\log N$.
\end{proposition}
\begin{proof}
    Suppose that $r\in\prod_{p\leqslant Q_1}(\mathbf{Z}/p\mathbf{Z})^\times$ is such that $$|A(r,(0))|=\max_{s\in \prod_p(\mathbf{Z}/p\mathbf{Z})^\times}|A(s,(0))|\geqslant (\log N)^4.$$ We recall that the function $F_A$ is given by
    $$F_A(x)=\sum_{a\in A}\sum_{n\geqslant 1}\frac{\chi(n)}{n}c( nax).$$
 We define
    \begin{equation*}
        \mathcal{P}_\textnormal{med}=\{p\in[(\log N)^{1/2},(\log N)^{20}]: p\text{ is prime}\}
    \end{equation*} and we shall think of these as `medium'-range primes. Correspondingly, we define
    \begin{equation*}
        \mathcal{R}_\textnormal{med}=\{n\geqslant 1:n\text{ is coprime to all primes in }\mathcal{P}_\textnormal{med}\}.
    \end{equation*} We fix throughout the parameters $Q_1=(\log N)^{1/2}$ and $Q=(\log N)^{20}$. In this proof, we will consider the function $$F_\textnormal{med}(x)=\sum_{a\in A}\sum_{n \in \mathcal{R}_\textnormal{med}}\frac{\chi(n)}{n}c( nax).$$ First, we show how one may obtain $F_\textnormal{med}$ from $F_A$ by `sifting' out all primes in $\mathcal{P}_\textnormal{med}$ with a procedure much like that in Proposition \ref{prop:basicestimate}.
    \begin{lemma}\label{lem:F_mmodel}
        We have that $\lVert F_\textnormal{med}\rVert_1\ll \lVert F_A\rVert_1$.
        
    \end{lemma}
    \begin{proof}[Proof of Lemma \ref{lem:F_mmodel}]
        We follow the proof of Proposition \ref{prop:basicestimate} to show that \begin{equation}\label{eq:F_m}\sum_{k|\prod_{p\in\mathcal{P}_\textnormal{med}}p}\frac{\mu(k)\chi(k)}{k}F_A(kx)=F_\textnormal{med}(x).
        \end{equation} As $\chi$ is multiplicative, we can calculate
    \begin{align*}
        \sum_{k|\prod_{p\in\mathcal{P}_\textnormal{med}}p}\frac{\mu(k)\chi(k)}{k}\sum_{n\geqslant 1}\frac{\chi(n)}{n}c (nkx)
        &=\sum_{m\geqslant 1}\frac{\chi(m)}{m}c (mx)\sum_{k|\gcd(m,\prod_{p\in\mathcal{P}_\textnormal{med}}p)}\mu(k)\\
        &= \sum_{m\in\mathcal{R}_\textnormal{med}}\frac{\chi(m)}{m}c( mx).
    \end{align*}
    Replacing $x$ by $ax$ and summing over $a\in A$ proves \eqref{eq:F_m}. We again follow the proof of Proposition \ref{prop:basicestimate} to bound
    \begin{align*}
        \lVert F_\textnormal{med}\rVert_1&=\left\lVert \sum_{k|\prod_{p\in\mathcal{P}_\textnormal{med}}p}\frac{\mu(k)\chi(k)}{k}F(kx)\right\rVert_1\\
        &\leqslant \lVert F\rVert_1\prod_{p\in\mathcal{P}_\textnormal{med}}(1+1/p)\\
        &\ll \lVert F\rVert_1,
    \end{align*}
    since $\sum_{p\in\mathcal{P}_\textnormal{med}}\frac{1}{p}=\sum_{p\in[(\log N)^{1/2},(\log N)^{20}]}\frac{1}{p}\ll 1$ by Mertens' estimate.
    \end{proof}
    We will show that $\lVert F_\textnormal{med}\rVert_1\gg \log\log N$ which, by the lemma above, then implies the desired lower bound $\lVert F_A\lVert_1\gg \log \log N$. In order for the information about $A(r,(0))$ to be exploited, we will use the following lemma.
    \begin{lemma}\label{lem:L^1projection}
    Let $h(x)=\sum_{n\in\mathbf{Z}}\hat{h}(n)e(nx)\in L^1(\mathbf{T})$ and let $q\in\mathbf{N}$ and $\ell\in \mathbf{Z}/q\mathbf{Z}$. Then $\lVert \sum_{n\equiv\ell\mdsublem q}\hat{h}(n)e(nx)\lVert_1\leqslant \lVert h\rVert_1$.
        
    \end{lemma}
    \begin{proof}[Proof of Lemma \ref{lem:L^1projection}]
        This follows as
        \begin{align*}
            \sum_{k\equiv \ell\mdsub q}\hat{h}(k)e(kx)=\frac{1}{q}\sum_{j=0}^{q-1}e(-\ell j/q)h(x+j/q),
        \end{align*}
        by orthogonality of characters modulo $q$. Hence, by the triangle inequality and as $\lVert h(x+j/q)\rVert_1=\lVert h(x)\rVert_1$, the left hand side has $L^1$-norm at most $\lVert h\rVert_1$.
    \end{proof}
    Hence, if we define $$\Proj(F_\textnormal{med};r,(0))(x)=\sum_{k\equiv r\mdsub{\prod_{p\leqslant Q_1} p}}\hat{F}_\textnormal{med}(k)e(kx)$$
    to be the function obtained by keeping only those terms in the Fourier series of $F_\textnormal{med}$ whose frequencies are $r\md {\prod_{p\leqslant Q_1} p}$, then $\lVert \Proj(F_\textnormal{med};r,(0))\rVert_1\leqslant\lVert F_\textnormal{med}\rVert_1$.
    We shall analyse the structure of $\Proj(F_\textnormal{med};r,(0))$. Recall that we use the notation $$A(s,(\mu_p))=\{a\in A: a\equiv s\prod_{p\leqslant Q_1} p^{\mu_p}\md{\prod_{p\leqslant Q_1} p^{\mu_p+1}}\}.$$ We can decompose
    \begin{equation}\label{eq:F_mdecompomodel}F_\textnormal{med}(x)=\sum_{s,(\mu_p)}\sum_{a\in A(s,(\mu_p))}\left(\sum_{n\in\mathcal{R}_\textnormal{med}}\frac{\chi(n)}{n}c (nax)\right)
    \end{equation}
    and we shall consider the contribution from each $A(s,(\mu_p))$ to $\Proj(F_\textnormal{med};r,(0))$. It is convenient to write
    $$F_{\textnormal{med},s,(\mu_p)}(x)\vcentcolon=\sum_{a\in A(s,(\mu_p))}\sum_{n\in \mathcal{R}_\textnormal{med}}\frac{\chi(n)}{n}c( nax)$$
    so that \eqref{eq:F_mdecompomodel} becomes
    \begin{equation*}\label{eq:F_mdecompomodel2}
        F_\textnormal{med}=\sum_{s,(\mu_p)}F_{\textnormal{med},s,(\mu_p)}.
    \end{equation*}
    Let $a\in A(s,(\mu_p))$ for some $(s,(\mu_p))$ and suppose that one of its terms $\frac{\chi(n)}{n}c( nax)=\frac{\chi(n)}{2n}(e(nax)+e(-nax))$ contributes to $\Proj(F_\textnormal{med};r,(0))$. This means precisely that $n\in\mathcal{R}_\textnormal{med}$ satisfies $$\pm na\equiv r\md{\prod_{p\leqslant Q_1} p}$$ and in particular this implies that $\mu_p=0$ and that $\nu_p(n)=0$ for all $p\leqslant Q_1$ (here $\nu_p(n)$ denotes the largest integer such that $p^{\nu_p}$ divides $n$). We deduce that $\Proj\bigl(F_{\textnormal{med},s,(\mu_p)};r,(0)\bigr)=0$ unless $(\mu_p)= (0)$, and so
    \begin{align}\label{eq:projdecompomodel}
        \Proj(F_\textnormal{med};r,(0))=\sum_{s\in \prod_{p\leqslant Q_1}(\mathbf{Z}/p\mathbf{Z})^\times}\Proj\bigl(F_{\textnormal{med},s,(0)};r,(0)\bigr).
    \end{align}
    We know from above that if a term $\frac{\chi(n)}{n}c( nax)$ contributes to $\Proj(F_\textnormal{med};r,(0))$, where $n\in\mathcal{R}_\textnormal{med}$ and $a\in A(s,(0))$ for some $s\in \prod_{p\leqslant Q_1}(\mathbf{Z}/p\mathbf{Z})^\times$, then $(n,p)=1$ for all $p\leqslant Q_1$. Hence, we must in fact have that $n\in\mathcal{R}_{Q}$, where we recall that $\mathcal{R}_Q$ is the set of $Q$-rough numbers. Finally, if $a\in A(s,(0))$ then $a\equiv s\md{\prod p}$ by definition, so the congruence $\pm an\equiv r\md{\prod p}$ is satisfied precisely when $n\equiv \pm rs^{-1}\md{\prod p}$. It follows that we can precisely write down the contribution of $F_{\textnormal{med},s,(0)}$ to $\Proj(F_\textnormal{med};r,(0))$ as\footnote{Intuitively, one should think of this projection onto a residue class $r\md{\prod_{p\leqslant Q_1}p}$ (for which $A(r,(0))$ is large) as replacing the need to `sift' out primes $p\leqslant Q_1$ as in Bourgain's approach in Proposition \ref{prop:basicestimate}, in that it also restricts the sum over $n$ to a sum over $Q$-rough numbers. This is ultimately the reason why we are able to gain a factor of $\log Q_1\gg \log  \log N$.}
    \begin{align}\label{eq:projcontmodel} 
    \Proj\bigl(F_{\textnormal{med},s,(0)};r,(0)\bigr)&=\frac{1}{2}\sum_{a\in A(s,(0))}\sum_{\substack{n\in\mathcal{R}_Q\\ n\equiv rs^{-1}\mdsub{\prod p}}}\frac{\chi(n)}{n}e(nax)\\
    &+\frac{1}{2}\sum_{a\in A(s,(0))}\sum_{\substack{n\in\mathcal{R}_Q \\ n\equiv -rs^{-1}\mdsub{\prod p}}}\frac{\chi(n)}{n}e(-nax).\nonumber
    \end{align}
    Our objective is to obtain a strong lower bound for the $L^1$-norm of $\Proj(F_\textnormal{med};r,(0))$. For this purpose, we will show that the main term comes from the projection of $F_{\textnormal{med},r,(0)}$, whereas all other $F_{\textnormal{med},s,(0)}$ contribute an error term which is negligible in an $L^2$-sense. In practice, we will find estimates for the $L^1$-norm of the convolution $\Proj(F_\textnormal{med};r,(0))*V_{T}=\Proj(F_\textnormal{med}*V_{T};r,(0))$ where $V_{T}(x)$ is a de la Vall\'ee-Poussin kernel, rather than for $\Proj(F_\textnormal{med};r,(0))$. Recall that $T=e^{O((\log N)^5)}$ is such that $A\subset [-T,T]$. We also recall that the \emph{de la Vall\'ee-Poussin kernel} is given by
$$V_{T}(x)=\sum_{n=-2T}^{-T-1}\left(1-\frac{|2T+n|}{T}\right)e(nx)+\sum_{n=-T}^Te(nx)+\sum_{n=T+1}^{2T}\left(1-\frac{|2T-n|}{T}\right)e(nx)$$  and that it satisfies the following basic properties (see \cite[Chapter VIII, (205)]{natanson}):
\begin{itemize}
    \item $\hat{V}_{T}(n)=1$ for $|n|<T$ while $\hat{V}_{T}(n)=0$ for $|n|\geqslant 2T$. 
    \item $\lVert V_{T}\rVert_1 \leqslant 3$.
\end{itemize}
These imply the following useful relation between $\Proj(F_\textnormal{med};r,(0))*V_{T}$ and $F_A$.
\begin{lemma}\label{lem:dlVPconvmodel}
    We define $F^*=\Proj(F_\textnormal{med};r,(0))*V_{T}$. Then $\lVert F^*\rVert_1\ll \lVert F_A\rVert_1$.
\end{lemma}
\begin{proof}[Proof of Lemma \ref{lem:dlVPconvmodel}]
    Note that from (a trivial instance of) Young's convolution inequality we obtain
    $$\lVert \Proj(F_\textnormal{med};r,(0))*V_{T}\rVert_1\leqslant \lVert \Proj(F_\textnormal{med};r,(0))\rVert_1\lVert V_{T}\rVert_1\ll \lVert F_\textnormal{med}\rVert_1$$
    where we used Lemma \ref{lem:L^1projection} and that $\lVert V_{T}\rVert _1\ll 1$. The result follows upon recalling Lemma \ref{lem:F_mmodel}.
\end{proof}
The following lemma provides properties of the functions $\Proj(F_{\textnormal{med},s,(0)};r,(0))$ appearing in the expression \eqref{eq:projdecompomodel} for $\Proj(F_\textnormal{med};r,(0))$ which are relevant for the eventual purpose of applying a McGehee-Pigno-Smith style construction to lower bound the $L^1$ norm.
\begin{lemma}\label{lem:allscontmodel}
    We have that
    \begin{align*}   &\sum_{s\in \prod_{p\leqslant Q_1}(\mathbf{Z}/p\mathbf{Z})^\times}\Proj\bigl(F_{\textnormal{med},s,(0)};r,(0)\bigr)(x)\\
    &=\frac{1}{2}\left(\sum_{a\in A(r,(0))}e(ax) + \sum_{a\in A(-r,(0))}e(-ax) + E_{(0)}(x)\right)
    \end{align*}
        for some function $E_{(0)}:\mathbf{T}\to\mathbf{C}$ satisfying $\lVert E_{(0)}*V_{T}\rVert_2\ll (\log N)^{-2}|A(r,(0))|^{1/2}$.
\end{lemma}
We write $E_{(0)}$ with the subscript $(0)$ to be consistent with the notation used in the proof of the general setting Proposition \ref{prop:distmodpstruct}.
\begin{proof}[Proof of Lemma \ref{lem:allscontmodel}]
This is immediate from an application of Lemma \ref{lem:generalcont} with the sets $B_s=A(s,(0))$ and $k=\pm 1$, and with our choice of parameters $Q=(\log N)^{20}$ and $T=e^{O((\log N)^5)}$, by noting that the explicit Fourier expansion \eqref{eq:projcontmodel} of the functions $\Proj\bigl(F_{\textnormal{med},s,(0)};r,(0)\bigr)$ is precisely of the type to which Lemma \ref{lem:generalcont} applies. Furthermore, we note that by assumption $K=\max_s|B_s|=\max_s|A(s,(0))|= |A(r,(0))|$.
\end{proof}
Let us summarise what we have achieved thus far. We have shown that if $A\subset \mathbf{Z}\setminus\{0\}\subset [-T,T]$ where $T\leqslant e^{O((\log N)^5)}$, and moreover $|A(r,(0))|= \max_s|A(s,(0))|$, then by \eqref{eq:projdecompomodel} and the lemma above, we can write the function $\Proj(F_\textnormal{med};r,(0))$ as
\begin{align*}
    \Proj\bigl(F_{\textnormal{med}};r,(0)\bigr)
    =\frac{1}{2}\left(\sum_{a\in A(r,(0))\cup -A(-r,(0))}e(ax) + E_{(0)}(x)\right),\nonumber
\end{align*}
and where
\begin{equation*}\label{eq:L^2bound}
\lVert E_{(0)}*V_{T}\rVert_2\ll (\log N)^{-2}|A(r,(0))|^{1/2}.
\end{equation*}
Let us finally consider the function $F^*(x)=\Proj(F_\textnormal{med};r,(0))*V_{T}$. By the above, we may write
\begin{align*}
    F^*(x)=\left(\frac{1}{2}\sum_{a\in A(r,(0))\cup-A(-r,(0))}e(ax)\right)*V_{T}+E^*(x)
\end{align*}
where $\lVert E^*\rVert_2\ll (\log N)^{-2}|A(r,(0))|^{1/2}$. Hence,
\begin{align*}
    F^*(x) = \frac{1}{2}\sum_{a\in A(r,(0))\cup-A(-r,(0))}e(ax)+E^*,
\end{align*}
noting that $A\subset [-T,T]$ and that $\hat{V}_T(n)=1$ whenever $|n|\leqslant T$ which implies that $e(\pm ax)*V_T(x)=e(\pm ax)$ for all $a\in A$.

\medskip

As we are assuming that $|A(r,(0))|\geqslant (\log N)^4$, it now follows from an application of the next theorem with $B_1=A(r,(0)), B_2=-A(-r,(0))$, $E=E^*$ and $K=\log N$ that such a function $F^*$ of the form above has $L^1$-norm at least $\lVert F^*\rVert_1\gg \log \log N$. Note that this implies that $\lVert F_A\rVert_1 \gg \log\log N$ by Lemma \ref{lem:dlVPconvmodel} and hence this finishes the proof of Proposition \ref{prop:distmodel1}.
\begin{theorem}\label{th:MPS1}
    Let $K\geqslant 1$ be a parameter. Let $B_1,B_2\subset \mathbf{Z}$ be finite with $|B_1|\geqslant K$, and let $E\in L^2(\mathbf{T})$ satisfy $\lVert E\rVert_2\leqslant |B_1|^{1/2}/K$. Then $\lVert \hat{1}_{B_1} + \hat{1}_{B_2} + E\rVert_1\gg \log K$.
\end{theorem}
\begin{proof}[Proof of Theorem \ref{th:MPS1}]
This follows from a relatively straightforward adaptation of the method of McGehee-Pigno-Smith, see Appendix A.
\end{proof}
\end{proof}
We now come to the main result of this section. Corollary \ref{cor:largeA(s,mu)} shows that either $S(A)\geqslant N/3+c(\log N)^{1/2}$ or else one of the sets $A(r,(\nu_p))$ has size at least $N^{1/2}$ (say). The previous proposition allows us to obtain the desired bound $S(A)\geqslant N/3+c\log\log N$ if it happens to be the case that $(\nu_p)=(0)$. We show in the following proposition that one can still deduce strong structural information about $A$ from $A(r,(\nu_p))$ being large for a general $(\nu_p)\in\mathbf{N}^{\pi(Q_1)}$. We first need to introduce some convenient notation and we shall write $$(\nu'_p)_{p\leqslant Q_1}\prec (\nu_p)_{p\leqslant Q_1}$$ if $\nu'_p\leqslant\nu_p$ for all $p\leqslant Q_1$ and $\nu'_p<\nu_p$ for at least one $p$. We shall also not repeatedly write $p\leqslant Q_1$ and it shall be clear from context what the range of $p$ is.
\begin{proposition}\label{prop:distmodpstruct}
    Let $A\subset\mathbf{Z}\setminus\{0\}$ be a set of size $N$ and assume that $A\subset [-T,T]$ where $T\ll e^{O((\log N)^5)}$. Then either $\lVert F_A\rVert_1 \gg \log\log N$ so that $S(A)\geqslant N/3+c\log\log N$, or else the following holds. Let $Q_1=(\log N)^{1/2}$. Suppose that there exist an $r\in \prod_{p\leqslant Q_1}\left(\mathbf{Z}/p\mathbf{Z}\right)^\times$ and $(\nu_p)_{p\leqslant Q_1}\in\mathbf{N}^{\pi(Q_1)}$ such that $$|A(r,(\nu_p))|=\max_{s\in\prod_{p\leqslant Q_1}(\mathbf{Z}/p\mathbf{Z})^\times}|A(s,(\nu_p))|\geqslant (\log N)^{4}.$$
    Then there exist an $r'\in \prod_{p\leqslant Q_1}\left(\mathbf{Z}/p\mathbf{Z}\right)^\times$ and $(\nu'_p)_{p\leqslant Q_1}\in\mathbf{N}^{\pi(Q_1)}$ such that 
    \begin{itemize}
        \item $(\nu'_p)\prec(\nu_p)$,
        \item $|A(r',(\nu'_p))|\geqslant (\log N)^{-4}|A(r,(\nu_p))|$.
    \end{itemize}
\end{proposition}
\begin{proof}
    We argue by assuming that there exist $r,(\nu_p)$ such that $\max_s|A(s,(\nu_p))|=|A(r,(\nu_p))|\geqslant (\log N)^4$ and that $|A(r',(\nu'_p))|< (\log N)^{-4}|A(r,(\nu_p))|$ for all $r'$ and $(\nu'_p)\prec(\nu_p)$. We shall show that under these assumptions, $\lVert F_A\rVert_1\gg \log\log N$ and recall from Proposition \ref{prop:basicestimate} that such a bound implies that $S(A)\geqslant N/3+c\log\log N$ for some absolute $c>0$.
    
    \medskip
    
    The proof has many similarities to the proof of the model setting above, so we shall be brief in our discussion of these parts. Let us consider, as per usual, the function
    $$F_A(x)=\sum_{a\in A}\sum_{n\geqslant 1}\frac{\chi(n)}{n}c( nax).$$
    We fix throughout the parameters $Q_1=(\log N)^{1/2}$ and $Q=(\log N)^{20}$. We define
    $\mathcal{P}_\textnormal{med}=\{p\in[Q_1,Q]: p\text{ is prime}\}$ and $\mathcal{R}_\textnormal{med}=\{n\geqslant 1:(n,p)=1\text{ for all }p\in\mathcal{P}_\textnormal{med}\}$.
    We consider the function $F_\textnormal{med}(x)=\sum_{a\in A}\sum_{n \in \mathcal{R}_\textnormal{med}}\frac{\chi(n)}{n}c( nax)$. Exactly as in Lemma \ref{lem:F_mmodel}, we have the following relation between the $L^1$ norms of $F_\textnormal{med}$ and $F_A$.
    \begin{lemma}\label{lem:F_m}
        We have that $\lVert F_\textnormal{med}\rVert_1\ll \lVert F_A\rVert_1$.
        
    \end{lemma}
    Combining this with Lemma \ref{lem:L^1projection} shows the following.
    \begin{lemma}\label{lem:L^1proj}
    Let $$\Proj(F_\textnormal{med};r,(\nu_p))(x)=\sum_{k\equiv r\prod_{p\leqslant Q_1} p^{\nu_p}\mdsublem{\prod_{p\leqslant Q_1} p^{\nu_p+1}}}\hat{F}_\textnormal{med}(k)e(kx)$$
    be the function obtained by keeping only those terms in the Fourier series of $F_\textnormal{med}$ whose frequencies are $r\prod_{p\leqslant Q_1} p^{\nu_p}\mdlem {\prod_{p\leqslant Q_1} p^{\nu_p+1}}$. Then $\lVert \Proj(F_\textnormal{med};r,(\nu_p))\rVert_1\ll \lVert F_A\rVert_1$.
        
    \end{lemma}
    We shall analyse the structure of $\Proj(F_\textnormal{med};r,(\nu_p))$. We can decompose
    \begin{equation}\label{eq:F_mdecompo}F_\textnormal{med}(x)=\sum_{s,(\mu_p)}\sum_{a\in A(s,(\mu_p))}\left(\sum_{n\in\mathcal{R}_\textnormal{med}}\frac{\chi(n)}{n}c( nax)\right)
    \end{equation}
    and we shall consider the contribution from each $A(s,(\mu_p))$ to $\Proj(F_\textnormal{med};r,(\nu_p))$. It is convenient to write
    $$F_{\textnormal{med},s,(\mu_p)}(x)\vcentcolon=\sum_{a\in A(s,(\mu_p))}\sum_{n\in \mathcal{R}_\textnormal{med}}\frac{\chi(n)}{n}c( nax)$$
    so that \eqref{eq:F_mdecompo} becomes
    \begin{equation}\label{eq:F_mdecompo2}
        F_\textnormal{med}=\sum_{s,(\mu_p)}F_{\textnormal{med},s,(\mu_p)}.
    \end{equation}
    Let $a\in A(s,(\mu_p))$ for some $(s,(\mu_p))$ and suppose that one of its terms $\frac{\chi(n)}{n}c( nax)=\frac{\chi(n)}{2n}(e(nax)+e(-nax))$ contributes to $\Proj(F_\textnormal{med};r,(\nu_p))$. This means precisely that $n\in\mathcal{R}_\textnormal{med}$ satisfies $$\pm na\equiv r\prod_{p\leqslant Q_1} p^{\nu_p}\md{\prod_{p\leqslant Q_1} p^{\nu_p+1}}$$ and in particular this implies that $\nu_p\geqslant \nu_p(a)=\mu_p$ and that $\nu_p(n)=\nu_p-\mu_p$ for all $p\leqslant Q_1$ (here $\nu_p(n)$ denotes the largest integer such that $p^{\nu_p}$ divides $n$). We deduce that $\Proj\bigl(F_{\textnormal{med},s,(\mu_p)};r,(\nu_p)\bigr)=0$ unless $(\mu_p)\preceq (\nu_p)$, meaning that the only terms in \eqref{eq:F_mdecompo2} which contribute to $\Proj\bigl(F_\textnormal{med};r,(\nu_p)\bigr)$ come from those $(s,(\mu_p))$ with $(\mu_p)\preceq (\nu_p)$ and so
    \begin{align}\label{eq:projdecompo}
        \Proj(F_\textnormal{med};r,(\nu_p))=\sum_{(s,(\mu_p)):(\mu_p)\preceq(\nu_p)}\Proj\bigl(F_{\textnormal{med},s,(\mu_p)};r,(\nu_p)\bigr).
    \end{align}
    Let us consider $(\mu_p)\preceq (\nu_p)$. We know from our discussion above that if a term $\frac{\chi(n)}{n}c( nax)$ contributes to $\Proj(F_\textnormal{med};r,(\nu_p))$, where $n\in\mathcal{R}_\textnormal{med}$ and $a\in A(s,(\mu_p))$ for some $s\in \prod_{p\leqslant Q_1}(\mathbf{Z}/p\mathbf{Z})^\times$, then $\nu_p(n)=\nu_p-\mu_p$ for all $p\leqslant Q_1$. Hence, we can write $n=n'\prod_{p\leqslant Q_1}p^{\nu_p-\mu_p}$ for some $n'\in\mathcal{R}_{Q}$, where we recall that $\mathcal{R}_Q$ is the set of $Q$-rough numbers. Finally, if $a\in A(s,(\mu_p))$ then $a\equiv s\prod p^{\mu_p}\md{\prod p^{\mu_p+1}}$ by definition, so the congruence $\pm an'\prod_{p\leqslant Q_1}p^{\nu_p-\mu_p}\equiv r\prod_{p\leqslant Q_1} p^{\nu_p}\md{\prod_{p\leqslant Q_1} p^{\nu_p+1}}$ is satisfied precisely when $n'\equiv \pm rs^{-1}\md{\prod p}$. It follows that for each $(s,(\mu_p))$ with $(\mu_p)\preceq(\nu_p)$, we can precisely write down the contribution of $F_{\textnormal{med},s,(\mu_p)}$ to $\Proj(F_\textnormal{med};r,(\nu_p))$ as
    \begin{align*}
        \Proj\bigl(F_{\textnormal{med},s,(\mu_p)};r,(\nu_p)\bigr)(x)&=\sum_{a\in A(s,(\mu_p))}\sum_{\substack{n=n'\prod p^{\nu_p-\mu_p}\\ n'\in\mathcal{R}_Q\\ n'\equiv rs^{-1}\mdsub{\prod p}}}\frac{\chi(n)}{2n}e(nax)\\
        &+ \sum_{a\in A(s,(\mu_p))}\sum_{\substack{n=n'\prod p^{\nu_p-\mu_p}\\ n'\in \mathcal{R}_Q \\ n'\equiv -rs^{-1}\mdsub{\prod p}}}\frac{\chi(n)}{2n}e(-nax),
    \end{align*}
    Using the multiplicative nature of $\chi$, we can take out the factor $k(\mu_p)\vcentcolon=\prod_{p\leqslant Q_1} p^{\nu_p-\mu_p}$ and simplify this further to
    \begin{align}\label{eq:projcont} 
    \Proj\bigl(F_{\textnormal{med},s,(\mu_p)};r,(\nu_p)\bigr)&=\frac{\chi(k(\mu_p))}{2k(\mu_p)}\sum_{a\in A(s,(\mu_p))}\sum_{\substack{n'\in\mathcal{R}_Q\\ n'\equiv rs^{-1}\mdsub{\prod p}}}\frac{\chi(n')}{n'}e(n'k(\mu_p)ax)\\
    &+\frac{\chi(k(\mu_p))}{2k(\mu_p)}\sum_{a\in A(s,(\mu_p))}\sum_{\substack{n'\in\mathcal{R}_Q \\ n'\equiv -rs^{-1}\mdsub{\prod p}}}\frac{\chi(n')}{n'}e(-n'k(\mu_p)ax).\nonumber
    \end{align}
    Contrary to the model case in Proposition \ref{prop:distmodel1}, these functions $\Proj\bigl(F_{\textnormal{med},s,(\mu_p)};r,(\nu_p)\bigr)$ can contribute to $\Proj(F_\textnormal{med};r,(\nu_p))$ for all $s\in\prod_{p\leqslant Q_1}(\mathbf{Z}/p\mathbf{Z})^\times$ and all $(\mu_p)\preceq(\nu_p)$, rather than only for $(\mu_p)=(\nu_p)$. For the purpose of obtaining a lower bound for the $L^1$-norm of $\Proj(F_\textnormal{med};r,(\nu_p))$, we will show that the main term still comes from the projection of $F_{\textnormal{med},r,(\nu_p)}$, whereas all other $F_{\textnormal{med},s,(\mu_p)}$ with $(\mu_p)\preceq (\nu_p)$ contribute an error term which is negligible in an $L^2$-sense due to our assumption that $\max_{(\mu_p)\prec(\nu_p)} \max_s|A(s,(\mu_p))|<(\log N)^{-4}|A(r,(\nu_p))|$. Again, we will in practice find estimates for the $L^1$-norm of the convolution $\Proj(F_\textnormal{med};r,(\nu_p))*V_{T}$ where $V_{T}(x)$ is a de la Vall\'ee-Poussin kernel, and $T=e^{O((\log N)^5)}$ is such that $A\subset [-T,T]$. Analogously to Lemma \ref{lem:dlVPconvmodel}, we note the following useful relation between $\Proj(F_\textnormal{med};r,(\nu_p))*V_{T}$ and $F_A$.
\begin{lemma}\label{lem:dlVPconv}
    We define $F^*=\Proj(F_\textnormal{med};r,(\nu_p))*V_{T}$. Then $\lVert F^*\rVert_1\ll \lVert F_A\rVert_1$.
\end{lemma}
Recall that we have an explicit expression \eqref{eq:projdecompo} for $\Proj(F_\textnormal{med};r,(\nu_p))$ in terms of the contributions $\Proj(F_{\textnormal{med},s,(\mu_p)};r,(\mu_p))$ which are given by \eqref{eq:projcont}, for $(\mu_p)\preceq (\nu_p)$. The following lemma provides a useful description of these $\Proj(F_{\textnormal{med},s,(\mu_p)};r,(\nu_p))$.
\begin{lemma}\label{lem:allscont}
    Let $(\mu_p)\preceq (\nu_p)$ and define $k(\mu_p)=\prod_{p\leqslant Q_1}p^{\nu_p-\mu_p}$. Then
    \begin{align*}   &\sum_s\Proj\bigl(F_{\textnormal{med},s,(\mu_p)};r,(\nu_p)\bigr)(x)\\
    &=\frac{\chi(k(\mu_p))}{2k(\mu_p)}\left(\sum_{a\in A(r,(\mu_p))}e(ak(\mu_p)x) + \sum_{a\in A(-r,(\mu_p))}e(-ak(\mu_p)x) + E_{(\mu_p)}(x)\right)
    \end{align*}
        for some function $E_{(\mu_p)}:\mathbf{T}\to\mathbf{C}$ satisfying $\lVert E_{(\mu_p)}*V_{T}\rVert_2\ll (\log N)^{-2}|A(r,(\nu_p))|^{1/2}$.
\end{lemma}
\begin{proof}[Proof of Lemma \ref{lem:allscont}]
This is immediate from an application of Lemma \ref{lem:generalcont} with $B_s=A(s,(\mu_p))$ and $k=\pm k(\mu_p)$, and with our choice of parameters $Q=(\log N)^{20}$ and $T=e^{O((\log N)^5)}$, by noting that the explicit Fourier expansion \eqref{eq:projcont} of the functions $\Proj\bigl(F_{\textnormal{med},s,(\mu_p)};r,(\nu_p)\bigr)$ is precisely of the type to which Lemma \ref{lem:generalcont} applies. Furthermore, we note that by assumption $K=\max_s|B_s|=\max_s|A(s,(\mu_p))|\leqslant |A(r,(\nu_p))|$ for all $(\mu_p)\preceq (\nu_p)$.
\end{proof}
Let us summarise what we have achieved thus far. We have shown that if $A\subset \mathbf{Z}\setminus\{0\}$ is a set of $N$ integers with $A\subset [-T,T]$ where $T\leqslant e^{O((\log N)^5)}$, and moreover $|A(r,(\nu_p))|\geqslant \max_s|A(s,(\mu_p))|$ for all $(\mu_p)\preceq(\nu_p)$, then by \eqref{eq:projdecompo} we can write the function $\Proj(F_\textnormal{med};r,(\nu_p))$ as
\begin{align}\label{eq:fullF_m}
        \Proj(F_\textnormal{med};r,(\nu_p))=\sum_{(\mu_p):(\mu_p)\preceq(\nu_p)}\sum_s\Proj\bigl(F_{\textnormal{med},s,(\mu_p)};r,(\nu_p)\bigr),
    \end{align}
and where there exists, for each $(\mu_p)\preceq(\nu_p)$, a function $E_{(\mu_p)}$ satisfying 
\begin{equation}\label{eq:L^2bound}
\lVert E_{(\mu_p)}*V_{T}\rVert_2\ll (\log N)^{-2}|A(r,(\nu_p))|^{1/2}
\end{equation}
such that
\begin{align}\label{eq:totalmu_pcont}
    &\sum_s\Proj\bigl(F_{\textnormal{med},s,(\mu_p)};r,(\nu_p)\bigr)(x)\\
    &=\frac{\chi(k(\mu_p))}{2k(\mu_p)}\left(\sum_{a\in A(r,(\mu_p))\cup -A(-r,(\mu_p))}e(ak(\mu_p)x) + E_{(\mu_p)}(x)\right),\nonumber
\end{align}
and we recall that $k(\mu_p)=\prod_{p\leqslant Q_1}p^{\nu_p-\mu_p}$ and that $s$ always ranges over $\prod_{p\leqslant Q_1}(\mathbf{Z}/p\mathbf{Z})^\times$.

\medskip

We are in addition assuming that $|A(r,(\nu_p))|=\max_s|A(s,(\nu_p))|\geqslant (\log N)^4$ while
\begin{equation}\label{eq:otherAsmall}\max_s|A(s,(\mu_p))|\leqslant (\log N)^{-4}|A(r,(\nu_p))|
\end{equation}
for all $(\mu_p)\prec (\nu_p)$. We show that this implies that out of all terms \eqref{eq:totalmu_pcont}, the contribution from $A(r,(\nu_p))\cup -A(-r,(\nu_p))$ turns out to form, for the purpose of obtaining a lower bound for $\lVert \Proj(F_\textnormal{med};r,(\nu_p))\rVert_1$, the main term. In practice, this means that we prove that the contribution of all other terms combined is negligible in an $L^2$-sense. To make this precise, we define
$$E(x)\vcentcolon= \frac{1}{2}E_{(\nu_p)}(x) + \sum_{(\mu_p)\prec(\nu_p)} \sum_{s}\Proj\bigl(F_{\textnormal{med},s,(\mu_p)};r,(\nu_p)\bigr)(x)$$
which by \eqref{eq:fullF_m} and \eqref{eq:totalmu_pcont} allows us to write
\begin{align}\label{eq:projwithE}
    \Proj(F_\textnormal{med};r,(\nu_p))(x)=\frac{1}{2}\sum_{a\in A(r,(\nu_p))\cup-A(-r,(\nu_p))}e(ax) + E(x),
\end{align} and we will prove the following $L^2$-estimate.
\begin{lemma}\label{lem:L^2eror}
    $E$ satisfies the estimate $\lVert E*V_{T}\rVert_2 \ll(\log N)^{-1}|A(r,(\nu_p))|^{1/2}$.
\end{lemma}
\begin{proof}[Proof of Lemma \ref{lem:L^2eror}]
   By \eqref{eq:L^2bound} we have the desired bound for $\lVert E_{(\nu_p)}*V_T\rVert_2$ so it suffices to consider the contribution of those terms \eqref{eq:totalmu_pcont} from all $(\mu_p)\prec(\nu_p)$ which we denote by \begin{align*}
   E'(x)&=\sum_{(\mu_p)\prec(\nu_p)}\frac{\chi(k(\mu_p))}{2k(\mu_p)}\left(\sum_{a\in A(r,(\mu_p))}e(ak(\mu_p)x)+\sum_{a\in A(-r,(\mu_p))}e(-ak(\mu_p)x)\right)\\
   &+\sum_{(\mu_p)\prec(\nu_p)}\frac{\chi(k(\mu_p))}{2k(\mu_p)}E_{(\mu_p)}(x).
   \end{align*}
   Let us denote the first term on the right hand side of the equation above by $E_1$ and the second $E_2$. The term $E_2$ is easy to estimate upon recalling \eqref{eq:L^2bound}, that $k(\mu_p)=\prod_{p\leqslant Q_1}p^{\nu_p-\mu_p}$ and that $\chi$ is $1$-bounded:
   \begin{align*}
       \lVert E_2*V_{T}\rVert_2&\leqslant \sum_{(\mu_p)\prec(\nu_p)}\frac{1}{2\prod_pp^{\nu_p-\mu_p}}\lVert E_{(\mu_p)}*V_{T}\rVert_2\\
       &\ll (\log N)^{-2}|A(r,(\nu_p))|^{1/2}\prod_{p\leqslant Q_1}(1+1/p+1/p^2+\dots)\\
       &\ll (\log N)^{-1}|A(r,(\nu_p))|^{1/2},
   \end{align*}
   where we used that $\prod_{p\leqslant Q_1}(1-1/p)^{-1}\ll \log Q_1\ll \log\log N$ by Mertens' theorem. For the term $E_1$, we may note by Parseval and \eqref{eq:otherAsmall} that $$\left\lVert \sum_{a\in A(r,(\mu_p))}e(ak(\mu_p)x)\right\rVert_2=|A(r,(\mu_p))|^{1/2}\leqslant (\log N)^{-2}|A(r,(\nu_p))|^{1/2}$$ for $(\mu_p)\prec (\nu_p)$ and that the same bound holds for the contribution from $A(-r,(\mu_p))$. By Young's inequality, this $L^2$ bound also holds after convolving with $V_{T}$ since $\lVert V_{T}\rVert_1\ll 1$:
   \begin{align*}
       \left\lVert \Big(\sum_{a\in A(r,(\nu_p))}e(ak(\mu_p)x)\Big)*V_{T}\right\rVert_2&\leqslant \left\lVert\sum_{a\in A(r,(\mu_p))}e(ak(\mu_p)x)\right\rVert_2\lVert V_{T}\rVert_1\\
       &\ll (\log N)^{-2}|A(r,(\nu_p))|^{1/2}.
   \end{align*} Hence, summing over all $(\mu_p)\prec(\nu_p)$ gives
\begin{align*}
\lVert E_1*V_{T}\rVert_2&\ll (\log N)^{-2}|A(r,(\nu_p))|^{1/2}\sum_{(\mu_p)\prec(\nu_p)}\frac{1}{\prod_p p^{\nu_p-\mu_p}}\\
&\ll (\log N)^{-2}|A(r,(\nu_p))|^{1/2}\prod_{p\leqslant Q_1}(1+1/p+1/p^2+\dots)\\
&\ll (\log N)^{-1}|A(r,(\nu_p))|^{1/2}.
   \end{align*}
\end{proof}

\medskip

Let us finally consider the function $F^*(x)=\Proj(F_\textnormal{med};r,(\nu_p))*V_{T}$. From \eqref{eq:projwithE} and Lemma \ref{lem:L^2eror}, we may write
\begin{align*}
    F^*(x)=\left(\frac{1}{2}\sum_{a\in A(r,(\nu_p))\cup-A(-r,(\nu_p))}e(ax)\right)*V_{T}+E^*(x)
\end{align*}
where $\lVert E^*\rVert_2\leqslant (\log N)^{-1}|A(r,(\nu_p))|^{1/2}$. Hence,
\begin{align*}
    F^*(x) = \frac{1}{2}\sum_{a\in A(r,(\nu_p))\cup-A(-r,(\nu_p))}e(ax)+E^*,
\end{align*}
noting that $A\subset [-T,T]$ and that $\hat{V}_T(n)=1$ whenever $|n|\leqslant T$ which implies that $e(\pm ax)*V_T(x)=e(\pm ax)$ for all $a\in A$.

\medskip

An application of Theorem \ref{th:MPS1} with $B_1=A(r,(\nu_p)), B_2=-A(-r,(\nu_p))$, $E=E^*$ and $K=\log N$ implies that that $\lVert F^*\rVert_1\gg \log\log N$ and by Lemma \ref{lem:dlVPconv} we deduce that $\lVert F_A\rVert_1\gg \log \log N$ which finishes the proof of Proposition \ref{prop:distmodpstruct}.
\end{proof}

\section{Non-Archimedean test functions}
In this section, we show that $F_A$ has large $L^1$-norm for sets $A$ which exhibit the strong structure provided by Proposition \ref{prop:distmodpstruct}. We shall achieve this by showing that $$c_A+R_Q=\sum_{a\in A}c(ax)+\sum_{a\in A}\sum_{1<n\in\mathcal{R}_Q}\frac{\chi(n)}{n}c(nax)$$ has large $L^1$-norm, where $Q=(\log N)^{20}$, and then quoting (iv) in Proposition \ref{prop:basicestimate}. We will obtain our bound for the $L^1$-norm by constructing a test function in a McGehee-Pigno-Smith style fashion, but with a crucial modification: instead of constructing the test function using an increasing ordering of the Fourier spectrum as in Appendix A, in particular relying on one-sided Fourier series, we construct the test function using an ordering of (a subset of) the spectrum based on $p$-adic valuations. First, we begin by showing that sets $A$ which have the properties that Corollary \ref{cor:largeA(s,mu)} and Proposition \ref{prop:distmodpstruct} establish contain the following convenient structure.
\begin{lemma}\label{lem:Astrongstructure}
    Let $A\subset\mathbf{Z}$ and assume that $A$ has the following two properties for a parameter $Q_1$.
    \begin{itemize}
        \item There exist $r^{*}\in\prod_{p\leqslant Q_1}(\mathbf{Z}/p\mathbf{Z})^\times$ and $(\nu_p^{*})_{p\leqslant Q_1}\in\mathbf{N}^{\pi(Q_1)}$ with $|A(r^{*},(\nu_p^{*}))|\geqslant N^{1/2}$.
        \item Whenever $r,(\nu_p)$ are such that $|A(r,(\nu_p))|=\max_{s\in \prod_{p\leqslant Q_1}(\mathbf{Z}/p\mathbf{Z})^\times}|A(s,(\nu_p))|$ and $|A(r,(\nu_p))|\geqslant (\log N)^4$, then there exist $r'$ and $(\nu'_p)\prec (\nu_p)$ satisfying $|A(r',(\nu'_p))|\geqslant (\log N)^{-4}|A(r,(\nu_p))|$.
    \end{itemize}
    Then we can find an integer $J\gg (\log N)/ \log \log N $, and a sequence of residues $r^{(i)}\in\prod_{p\leqslant Q_1}(\mathbf{Z}/p\mathbf{Z})^\times$ and $(\nu_p^{(i)})\in\mathbf{N}^{\pi(Q_1)}$ such that:
    \begin{itemize}
        \item [(i)] $|A(r^{(i)},(\nu_p^{(i)}))|=\max_s |A(s,(\nu_p^{(i)}))|$ for all $i\in[J]$,
        \item [(ii)] $|A(r^{(i+1)},(\nu_p^{(i+1)}))|\geqslant (\log N)^{4}|A(r^{(i)},(\nu_p^{(i)}))|$ for all $i<J$,
        \item [(iii)] $(\nu_p^{(i+1)})\succ (\nu_p^{(i)})$ for all $i<J$.
    \end{itemize}
\end{lemma}
\begin{proof}
    Let $r_1^*$ be such that $|A(r_1^*,(\nu_p^*))|=\max_s|A(s,(\nu_p^*))|$ so that $|A(r_1^*,(\nu_p^*))|\geqslant |A(r^*,(\nu_p^*))|\geqslant N^{1/2}$ by assumption. Let us take $r^{(J)}=r^*_1$ and $(\nu_p^{(J)})=(\nu_p^*)$. This choice then clearly satisfies condition (i) as well as that $|A(r^{(J)},(\nu_p^{(J)}))|\geqslant (\log N)^{8J}$ for some $J\gg (\log N )/\log \log N$. Suppose now that we have created, for some $j\geqslant 1$, a sequence of $r^{(i)}$ and $(\nu_p^{(i)})$ for $j\leqslant i\leqslant J$ which satisfies (i), (ii) and (iii) and has the additional property that $|A(r^{(j)},(\nu_p^{(j)}))|\geqslant (\log N)^{8j}$. By the assumptions of the lemma, there must exist some $(\mu_p)\prec(\nu_p^{(j)})$ so that $\max_s|A(s,(\mu_p))|\geqslant (\log N)^{-4}|A(r^{(j)},(\nu_p^{(j)}))|\geqslant (\log N)^{8j-4}$, and we may assume in addition that $(\mu_p)$ is a minimal element with respect to the partial order $\prec$ which satisfies the inequality $\max_s|A(s,(\mu_p))|\geqslant (\log N)^{8j-4}$. 
    
    \medskip
    
    We choose $s_0$ such that $|A(s_0,(\mu_p))|=\max_s|A(s,(\mu_p))|$. Hence, $|A(s_0,(\mu_p))|\geqslant (\log N)^{8j-4}$. Since $j\geqslant 1$, the assumptions of the lemma again imply that there exists some $(\mu_p')\prec(\mu_p)$ with $$\max_s|A(s,(\mu_p'))|\geqslant (\log N)^{-4}|A(s_0,(\mu_p))|\geqslant(\log N)^{8(j-1)}.$$ Let $s_0'$ be chosen such that $|A(s_0',(\mu_p'))|=\max_s|A(s,(\mu_p'))|$. We now claim that taking $r^{(j-1)}=s_0'$ and $(\nu_p^{(j-1)})=(\mu_p')$ works. Clearly, (i) is satisfied and (iii) holds as $(\mu_p')\prec (\mu_p)\prec (\nu_p^{(j)})$. Note also that $|A(s_0',(\mu_p'))|=\max_s|A(s,(\mu_p'))|\geqslant (\log N)^{8(j-1)}$. Finally, the claimed upper bound $|A(s_0',(\mu_p'))|<(\log N)^{-4}|A(r^{(j)},(\nu_p^{(j)}))|$ in (ii) follows by the choice of $(\mu_p)$ (being a minimal element with respect to $\prec$ which satisfies $\max_s|A(s,(\mu_p))|\geqslant (\log N)^{8j-4}$) and because $(\mu_p')\prec(\mu_p)$.
\end{proof}
Corollary \ref{cor:largeA(s,mu)} and Proposition \ref{prop:distmodpstruct} show that if $A\subset\mathbf{Z}\setminus\{0\}$ is a set of size $N$ and $A\subset [-T,T]$ where $T\leqslant e^{O(\log N)^5}$, then either
\begin{itemize}
    \item $\lVert F_A\rVert_1\gg \log \log N$,
    \item or $A$ satisfies the assumptions of Lemma \ref{lem:Astrongstructure} with $Q_1=(\log N)^{1/2}$.
\end{itemize}
If the first alternative above holds, then the conclusion of Theorem \ref{th:mainL^1} follows, so it only remains to prove this conclusion assuming that the following properties from the conclusion of Lemma \ref{lem:Astrongstructure} are satisfied. There exist $r^{(i)}\in\prod_{p\leqslant Q_1}(\mathbf{Z}/p\mathbf{Z})^\times$ and $(\nu_p^{(i)})\in\mathbf{N}^{\pi(Q_1)}$ such that 
\begin{itemize}
    \item [(i)] $|A(r^{(i)},(\nu_p^{(i)}))|=\max_s|A(s,(\nu_p^{(i)}))|$,
    \item [(ii)] $|A(r^{(i+1)},(\nu_p^{(i+1)}))|\geqslant (\log N)^4|A(r^{(i)},(\nu_p^{(i)}))|$,
    \item [(iii)] $(\nu_p^{(i+1)})\succ (\nu_p^{(i)})$,
\end{itemize}
for all $i\in[J]$ and for some integer $J\gg (\log N)/\log \log N$. We show that under these assumptions, $\lVert F_A\rVert_1\gg J/\log\log N\gg(\log N)/(\log\log N)^{2}$ and this finishes the proof because such a bound also confirms Theorem \ref{th:mainL^1} (in fact with a significantly stronger bound).

\medskip

We take $Q=(\log N)^{20}$ and recall, using (iv) in Proposition \ref{prop:basicestimate}, that
$$\lVert F_A\rVert_1\gg \lVert c_A+R_Q\rVert_1/\log \log N,$$
where $c_A=\sum_{a\in A}c( ax)$ and $$R_Q(x)=\sum_{a\in A}\sum_{1<n\in \mathcal{R}_Q}\frac{\chi(n)}{n}c( nax)$$ and where $\mathcal{R}_Q$ is the set of $Q$-rough numbers. Hence, to complete the proof of Theorem \ref{th:mainL^1}, it suffices to show that 
\begin{equation}\label{eq:goodL^1}
\lVert c_A+R_Q\rVert_1 \gg  J
\end{equation}
assuming the existence of the sets $A(r^{(i)},(\nu_p^{(i)}))$ satisfying (i), (ii) and (iii) above.
To do this, we will use a test function $\Phi$ which satisfies $\langle c_A+R_Q,\Phi\rangle \gg J$ and $\lVert \Phi\rVert_\infty\ll 1$. Such a test function $\Phi$ will be constructed by `going up' in the residue classes $r^{(i)}\prod_{p\leqslant Q_1}p^{\nu_p^{(i)}}\md{\prod p^{\nu_p^{(i)}+1}}$. In order for such a construction to work, we analyse again what the relevant projections onto each of these residue classes look like. Since we have obtained $c_A+R_Q$ from $F_A$ by `sifting' out all primes below $Q=(\log N)^{20}$, rather than only those in $((\log N)^{1/2},(\log N)^{20}]t$ as in the proof of Proposition \ref{prop:distmodpstruct}, these projections are somewhat simpler. We claim that
\begin{align}\label{eq:projc_A+R_Q}
    &\Proj\bigl(c_A+R_Q;r^{(i)},(\nu_p^{(i)})\bigr)(x)=\frac{1}{2}\sum_{a\in A(r^{(i)},(\nu_p^{(i)}))\cup-A(-r^{(i)},(\nu_p^{(i)}))}e(ax)\\ &\!+\!\sum_s\!\sum_{a\in A(s,(\nu_p^{(i)}))}\!\left(\sum_{\substack{1<n\in\mathcal{R}_Q\\n\equiv rs^{-1}\mdsub{\prod p}}}\!\frac{\chi(n)}{2n}e(anx)\!+\sum_{\substack{1<n\in\mathcal{R}_Q\\n\equiv -rs^{-1}\mdsub{\prod p}}}\!\frac{\chi(n)}{2n}e(-anx)\right).\nonumber
\end{align}
These projections of $c_A+R_Q$ are simpler to analyse than those of $F_\textnormal{med}$ in the proof of Proposition \ref{prop:distmodpstruct} in the sense that, here, only those $A(s,(\mu_p))$ with $(\mu_p)=(\nu_p^{(i)})$ contribute, rather than all $A(s,(\mu_p))$ with $(\mu_p)\preceq(\nu_p^{(i)})$. To see why \eqref{eq:projc_A+R_Q} holds, note that $\Proj(c_A;r^{(i)},(\nu_p^{(i)}))$ is precisely the first term on the right hand side. Further, if $a\in A$ and $1<n\in\mathcal{R}_Q$ are such that $\frac{\chi(n)}{n}c( nax)$ contributes to $\Proj(R_Q;r^{(i)},(\nu_p^{(i)}))$, then $na\equiv \pm r^{(i)}\prod_{p\leqslant Q_1}p^{\nu_p^{(i)}}\md{\prod_{p\leqslant Q_1}p^{\nu_p^{(i)}+1}}$ and hence, since $n$ is $Q$-rough and $Q_1=(\log N)^{1/2}<Q$, we need that $a\in \bigcup_s A(s,(\nu_p^{(i)}))$. We finally observe that if $a\in A(s,(\nu_p^{(i)}))$, then $\frac{\chi(n)}{n}c( nax)$ contributes if and only if $ns\equiv \pm r^{(i)}\md{\prod_{p\leqslant Q_1}p}$. 

\medskip

The second term on the right hand side of \eqref{eq:projc_A+R_Q} is of a form that we have studied in Lemma \ref{lem:generalcont} which in this situation states exactly that we may write
\begin{align}\label{eq:niceprojc_A+R_Q}
    \Proj(c_A+R_Q;r^{(i)},(\nu_p^{(i)}))=\frac{1}{2}\sum_{a\in A(r^{(i)},(\nu_p^{(i)}))\cup-A(-r^{(i)},(\nu_p^{(i)}))}e(ax) + E^{(i)}(x),
\end{align} for each $i\in [J]$ and where 
\begin{align*}
    \lVert E^{(i)}*V_T\rVert_2\ll(\log N)^{-2}\left(\max_s|A(s,(\nu_p^{(i)}))|\right)^{1/2}=(\log N)^{-2}|A(r^{(i)},(\nu_p^{(i)}))|^{1/2},
\end{align*}
by property (i) of these sets $A(r^{(i)},(\nu_p^{(i)}))$. Since we are also assuming that $A\subset[-T,T]$, we have that $c_A*V_T=c_A$ so that Young's convolution inequality (and that $\lVert V_T\rVert_1\ll 1$) gives $$\lVert c_A+R_Q\rVert_1\gg \lVert (c_A+R_Q)*V_T\rVert_1=\lVert c_A+R_Q*V_T\rVert_1.$$  Our final task was to prove \eqref{eq:goodL^1} assuming the existence of the sets $A(r^{(i)},(\nu_p^{(i)}))$ which satisfy (i), (ii) and (iii) above. It therefore suffices to show that $\lVert c_A+R_Q*V_T\rVert_1\gg  J$ and this we will deduce from the following theorem.
\begin{theorem}[non-Archimedean variant of the McGehee-Pigno-Smith construction]
Let $B\subset \mathbf{Z}$ be finite and suppose that for each $i\in[J]$ there exist $q_i\in\mathbf{N}$ and $r_i\in\mathbf{Z}/q_i\mathbf{Z}$ for which the following conditions hold.
\begin{itemize}
    \item [(i)] Let $B_i=B\cap\{n\in\mathbf{Z}:n\equiv r_i\mdlem{q_i}\}$. Assume that $|B_{i+1}|>10|B_i|$.
    \item [(ii)] Let $q_1|q_2|\dots|q_J$ and assume that the residue classes $\{n\in\mathbf{Z}:n\equiv r_i\mdlem{q_i}\}$ are pairwise disjoint.
\end{itemize}
If $E\in L^1(\mathbf{T})$ is a function such that $\lVert \Proj(E; r_i\mdlem{q_i})\rVert_2\leqslant |B_i|^{1/2}/10$, then $\lVert \hat{1}_B+E\rVert_1\gg  J$.
\end{theorem}
\begin{proof}
    We employ the same basic McGehee-Pigno-Smith construction as per usual and take $$g_i:\mathbf{Z}\to\mathbf{C}:g_i(n)=|B_i|^{-1}1_{n\in B_i}.$$
    Then we define $Q_i(x)=e^{-|\hat{g}_i(x)|}$ and $$\Phi_j(x)=\hat{g}_j+\hat{g}_{j-1}Q_j+\dots+\hat{g}_1Q_2\dots Q_j.$$ Lemma \ref{lem:basicMPS} states that $\lVert \Phi_j\rVert_1\leqslant 10$ for all $j$. We shall need to prove some new properties of these test functions relating to the divisibility properties of its Fourier spectrum.
    \begin{lemma}\label{lem:modularMPS}
We have that $\supp\left(\hat{g}_iQ_{i+1}\dots Q_j\right)^\wedge \subset \{n\in \mathbf{Z}: n\equiv r_i\mdlem{q_i}\}$ for all $j>i$.
\end{lemma}
\begin{proof}[Proof of Lemma \ref{lem:modularMPS}]
By definition, $\supp(g_i)=B_i$ and hence $\supp(\hat{g}_iQ_{i+1}\dots Q_j)^\wedge\subset B_i+\sum_{k\in[i+1,j)}\supp(\hat{Q}_k)$. It therefore suffices to show that $\supp(\hat{Q}_k)\subset q_k\cdot\mathbf{Z}$ for all $k\in [J]$ since $q_1|q_2|\dots|q_J$. For this, we can simply observe that $\hat{g}_k(x+1/q_k)=e(r_k/q_k)\hat{g}_k(x)$ because $\supp(g_k)=B_k\subset\{n:n\equiv r_k\md{q_k}\}$ by assumption. Hence, $|\hat{g}_k|$ is a $1/q_k$-periodic function and so is $Q_k=e^{-|\hat{g}_k|}$, implying that the Fourier spectrum of $Q_k$ consists of multiples of $q_k$ only.
\end{proof}
Since $\lVert \Phi_J\rVert_1\leqslant 10$, we can obtain a lower bound
\begin{align*}
\lVert \hat{1}_B+E\rVert_1&\gg \langle \hat{1}_B+E,\Phi_J\rangle\\
&=\sum_{j=1}^J\langle \hat{1}_B, \hat{g}_j\rangle - \sum_{1\leqslant j<k\leqslant J}\langle \hat{1}_B,\hat{g}_jQ_{j+1}\dots Q_{k-1}(1-Q_k)\rangle + \langle E,\Phi_J\rangle.
\end{align*}
By Parseval, $\langle \hat{1}_B, \hat{g}_j\rangle = 1$ for all $j$ so the first term above contributes $J$. By Lemma \ref{lem:modularMPS}, we have $$\langle \hat{1}_B,\hat{g}_jQ_{j+1}\dots Q_{k-1}(1-Q_k)\rangle =\langle \Proj(\hat{1}_B; r_j\md{q_j}),\hat{g}_jQ_{j+1}\dots Q_{k-1}(1-Q_k)\rangle$$ and hence, using that by Lemma \ref{lem:basicMPS} we have the inequalities $|Q_{j'}|\leqslant 1$, $|1-Q_k|\leqslant |\hat{g}_k|$, and $|\hat{g}_j|\leqslant 1$, we can use Cauchy-Schwarz to bound the second term by
\begin{align*}
    \sum_{1\leqslant j<k\leqslant J} \lVert \Proj(\hat{1}_B;r_j\md{q_j})\rVert_2\lVert \hat{g}_k\rVert_2 &= \sum_{1\leqslant j<k\leqslant J} |B_j|^{1/2}|B_k|^{-1/2}\\
    &\leqslant \sum_{1\leqslant j<k\leqslant J} 10^{-(k-j)/2}\leqslant J/(10^{1/2}-1),
\end{align*}
where we used the assumption that $|B_{i+1}|>10|B_i|$. Hence, $\lVert \hat{1}_B+E\rVert_1\gg J/2-\langle E,\Phi_J\rangle$.

\medskip

To finish the proof, we therefore only need to show that $\langle E,\Phi_J\rangle \leqslant J/10$. This will follow from the fact that $\langle E,\hat{g}_jQ_{j+1}\dots Q_J \rangle$ has size at most $1/10$ for all $j$. One can prove this by a single application of Cauchy-Schwarz:
\begin{align*}|\langle E, \hat{g}_jQ_{j+1}\dots Q_{J}\rangle|&=|\langle \Proj(E; r_j\md{q_j}), \hat{g}_jQ_{j+1}\dots Q_J\rangle|\\
&\leqslant \lVert \Proj(E; r_j\md{q_j})\rVert_2\lVert \hat{g}_j\rVert_2\leqslant 1/10\end{align*} where we used that $\lVert \hat{g}_j\rVert_2=|B_j|^{-1/2}$ and that $\lVert \Proj(E; r_j\md{q_j})\rVert_2 \leqslant |B_j|^{1/2}/10$ by assumption.

\end{proof}
We apply this theorem with the set $B=A\cup -A$, moduli $q_i= \prod_{p\leqslant Q_1} p^{\nu_p^{(i)}+1}$, and residues $r_i=r^{(i)} \prod_{p \leqslant Q_1}p^{\nu_p^{(i)}}$. We also take $E= R_Q*V_T$ so that by \eqref{eq:niceprojc_A+R_Q} and the inequality right after, we get that
\begin{align*}
    B_i&= A(r^{(i)}, (\nu_p^{(i)})) \cup -A(-r^{(i)},(\nu_p^{(i)}))\\
    \lVert \Proj(E; r_i\md{q_i})\rVert_2 &=\lVert E^{(i)}*V_T\rVert_2\ll (\log N)^{-2}|B_i|^{1/2}
\end{align*} and note furthermore that $q_1|q_2|\dots|q_J$ by property (iii). It is also immediate from property (ii) that $|B_{i+1}|\gg (\log N)^4|B_i|$ for all $i<J$. The theorem above now confirms the claimed bound \eqref{eq:goodL^1}: $\lVert c_A+R_Q\rVert_1\gg J\gg (\log N)/\log\log N$.
\section{The global structure of sets with $S(A)\leqslant N/3+C$}
The methods that we have introduced to prove Theorem \ref{th:main} can be exploited further to provide structural information about sets of integers $A$ with $S(A)\leqslant N/3+C$ for values of $C$ much larger than $\log \log N$. In this section, we shall focus on proving Theorem \ref{th:mainstructure}, although there also are other types of `structure' that one may provably find in $A$.
\begin{proposition}
    Let $A\subset\mathbf{Z}\setminus \{0\}$ have size $N$ and let $S(A)\leqslant N/3 +C$. Then there exists a set $B$ which is $F_4$-isomorphic to $A$ and which is contained in $[-T,T]$ where $T\leqslant N^{C^{O(1)}}$. Moreover, every subset $X\subset A$ satisfies the energy bound $E(X)\gg C^{-O(1)}\frac{|X|^4}{N}$.
\end{proposition}
\begin{proof}
By Corollary \ref{cor:densemodel}, we may find an $F_4$-isomorphic copy $A'$ of $A$ with $A'\subset[-T_1,T_1]$ where $T_1\leqslant e^{O((C\log N)^4)}$. By taking $Q=(C\log N)^{10}$ (say) in Proposition \ref{prop:basicestimate}, we find that $\lVert c_{A'}+R_Q\rVert_1 \ll (\log Q)\lVert F_{A'}\rVert_1\ll (\log Q)(S(A')-N/3)\ll C^2$ where we have used in the final inequality that $\log Q\ll C$ since $C\gg \log \log N$. In fact, it will be convenient to note that the exact same proof from Proposition \ref{prop:basicestimate} provides an analogous bound for every $L^p$-norm:
$$\lVert c_{A'}+R_Q\rVert_p \ll (\log Q)\lVert F_{A'}\rVert_p.$$ We know from Lemma \ref{lem:coefbound1} that $$R_Q(x)=\sum_{a\in A'}\sum_{1<n\in \mathcal{R}_Q}\frac{\chi(n)}{n}c( nax)$$ satisfies the bound $|\hat{R}_Q(m)|\ll (\log T)Q^{-1}\ll (C\log N)^{-5}$ for its Fourier coefficients with frequencies $m\in [-T_1,T_1]$. In particular, if we define $f(n)\vcentcolon = 1_{A'}(n)+1_{-A'}(n)+2\hat{R}_Q(n)$, then we have shown that 
\begin{itemize}
    \item $\lVert \hat{f}\rVert_1 \ll C^2$,
    \item $f(a)\geqslant 1/2$ for all $a\in A'$. 
\end{itemize}
One can also see by taking $p=\infty$ in the $L^p$-inequality above and by the trivial bound $|F_{A'}(x)|\leqslant N\max_x |\phi(x)-1/3|\leqslant N$ that $\lVert c_{A'}+R_Q\rVert_\infty \ll N C$, and hence we certainly have $\lVert \hat{f}\rVert_\infty \ll N^2$. Theorem \ref{th:smallL^1smalldim} states under these assumptions on $f$ that $\dim(A')\ll C^4(\log N)$. One may now simply use Theorem \ref{th:densemodel} to find the desired `dense' $F_4$-isomorphic copy $B\subset [-N^{C^{O(1)}},N^{C^{O(1)}}]$ (again using that $\log \log N\ll C$).

\medskip

Let $\psi:A\to A'$ be the $F_4$-isomorphism from above. For any subset $X\subset A$ we observe that $E(X)=E(\psi(X))$ precisely because $\psi$ preserves all additive relations of length at most $4$. We write $X'=\psi(X)$ and to establish the final part of this proposition, it suffices to show that $E(X')\gg C^{-O(1)}\frac{|X'|^4}{N}$. Note that the function $f=1_{A'}+1_{-A'}+2\hat{R}_Q$ from above retains its crucial properties after convolving with the de la Vall\'ee-Poussin kernel $V_{T_1}$:
\begin{itemize}
    \item $\lVert \hat{f}*V_{T_1}\rVert_1\ll C^2$,
    \item $(\hat{f}*V_{T_1})^\wedge(a)\geqslant 1/2$ for all $a\in A'$, since $A'\subset[-T_1,T_1]$.
\end{itemize}
Hence, $\hat{f}*V_{T_1}$ satisfies the assumptions of Corollary \ref{cor:smallL^1largeenergy} and we deduce that $$E(X')\gg C^{-O(1)}\frac{|X'|^4}{\lVert \hat{f}*V_{T_1}\rVert_2^2}.$$ Hence, our final task is to find a good bound for the $L^2$-norm of $(c_{A'}+R_Q)*V_{T_1}=c_{A'}+R_Q*V_{T_1}$. By Parseval, $\lVert c_{A'}\rVert_2\ll N^{1/2}$. Finally, an even stronger bound $\lVert R_Q*V_{T_1}\rVert_2\ll (\log N)^{-2}N^{1/2}$ may be deduced from Lemma \ref{lem:generalcont} (with $Q_1=1$).
\end{proof}
With a result like the proposition above in hand, which shows that any large subset of $A$ has large additive energy, one can invoke the standard tools of additive combinatorics to obtain Theorem \ref{th:mainstructure}; we briefly indicate how this is done. We need two central results from additive combinatorics. The first is the Balog-Szemer\'edi-Gowers Theorem, see \cite[Theorem 2.27]{taovubook}
\begin{theorem}[Balog-Szemer\'edi-Gowers Theorem]\label{th:balog}
    Let $K\geqslant 1$ and let $B\subset\mathbf{Z}$ have additive energy $E(B)\geqslant |B|^3/K$. Then there exists a subset $B'\subset B$ of size $|B'|\gg |B|/K^{O(1)}$ with small doubling $|B'-B'|\ll K^{O(1)}|B|$.
\end{theorem}
Freiman's theorem \cite[Theorem 5.44]{taovubook} describes the structure of sets $B$ with small doubling, stating that they must essentially be dense subsets of generalised arithmetic progressions. A generalised arithmetic progression of dimension $d$ is any set of the form $P=\{x_0+\sum_{j=1}^d n_jx_j:n_j\in\{0,1,\dots,L_j-1\}\}$. We say that $P$ is a proper progression if all the elements $x_0+\sum_{j=1}^d n_jx_j$ are pairwise distinct for $n_j\in\{0,1,\dots,L_j-1\}$, in which case $P$ has size $|P|= \prod_j L_j$. 
\begin{theorem}[Freiman's Theorem]\label{th:freiman}
    Let $K\geqslant 1$ and let $B\subset\mathbf{Z}$ be a set with doubling $|B-B|\leqslant K|B|$. Then there exists a proper generalised arithmetic progression $P$ of dimension $d$ such that $B\subset P$ and such that
    \begin{itemize}
        \item the dimension of $P$ is bounded by $d\leqslant K^{O(1)}$,
        \item $B$ is `dense' in $P$ in the sense that $|B|\geqslant e^{-K^{O(1)}}|P|$.
    \end{itemize}
\end{theorem}
We have not stated either theorem with the best currently known quantitative dependence on $K$.
\begin{proof}[Proof of Theorem \ref{th:mainstructure}]
Suppose that we have obtained a partial structured decomposition $A=(\cup_{i<j}A_i)\cup B$ where each $A_i$ has size $|A_i|\gg (CK)^{-O(1)}N$ and doubling $|A_i-A_i|\ll (CK)^{O(1)}$. Freiman's Theorem implies that $A_i$ is contained in some generalised arithmetic progression $P_i$ of dimension at most $(CK)^{O(1)}$ and size $|P_i|\ll e^{(CK)^{O(1)}}|A_i|$. The set $B$ either has size $|B|<(CK)^{-10}N$ in which case we have found the desired decomposition of $A$, or else the energy bound from the proposition above yields
$$E(B)\gg C^{-O(1)}\frac{|B|^4}{N}\gg (CK)^{-O(1)}|B|^3.$$ The Balog-Szemer\'edi-Gowers theorem tells us that $B$ contains a subset $B'$ of size $|B'|\geqslant (CK)^{-O(1)}|B|\gg (CK)^{-O(1)}N$ with small doubling $|B'-B'|\ll (CK)^{O(1)}$. So we take $A_j\vcentcolon= B'$ and we have obtained a new decomposition $A=(\cup_{i\leqslant j}A_i)\cup B^*$, where $A_j$ satisfies the same properties as the $A_i$ with $i<j$ and where $|B|-|B^*|\gg (CK)^{-O(1)}N$. Clearly this process terminates (in fact after at most $(CK)^{O(1)}$ steps), giving the desired decomposition of $A$.
\end{proof}
\appendix
\section{The McGehee-Pigno-Smith test function}
The purpose of this appendix is to provide a proof of Theorem \ref{th:MPS1}. 
\begin{proof}[Proof of Theorem \ref{th:MPS1}]
    Let us define $f(n)=1_{B_1}+1_{B_2}$ so that we aim to show that $\lVert \hat{f}+E\rVert_1\gg \log K$ under the assumption that $\lVert E\rVert_2\leqslant |B_1|^{1/2}/K$. To prove this, we use the method of McGehee-Pigno-Smith to construct a test function $\Phi$ satisfying $\lVert \Phi\rVert_\infty \ll 1$ and $\langle \Phi, \hat{f}+E\rangle\gg \log K$.

    \medskip
    
Let us define the sets $A_1,A_2,\dots,A_J$ to be subsets of $B_1\cup B_2$ satisfying the following:
    \begin{itemize}
        \item $A_i$ consists of the $100^i\lfloor|B_1\cup B_2|/K^2\rfloor$ smallest integers in $(B_1\cup B_2)\setminus\left(\cup_{j=1}^{i-1}A_i\right)$. In particular, $|A_{i+1}|=100|A_{i}|$ for all $i\in [J]$.
        \item $J\gg \log K$.
    \end{itemize}
    We define for each $i\in [J]$ the basic function $g_i:A_i\to\mathbf{C}$ by $g_i(n)=|A_i|^{-1}1_{\{n\in A_i\}}$. As usual, this has the following properties
\begin{align}
    \supp(g_i)&\subseteq A_i\label{eq:inneprod}\\
    \lVert \hat{g}_i\rVert_\infty&\leqslant 1\nonumber\\
    \langle \hat{f},\hat{g}_i\rangle &\geqslant 1\nonumber.    
\end{align}
The function $|\hat{g}_i(x)|=\sum_{n\in\mathbf{Z}}c_i(n)e(nx)$ is even so has a Fourier series with $c_i(n)=c_i(-n)$. We then define for each $i\in [J]$ the following function
\begin{equation*}
    h_i(x)=c_i(0)+2\sum_{n<0}c_i(n)e(nx)
\end{equation*}
and we also define the `correction' function $Q_i(x)=\exp(-h_i(x))$ and we note the following properties.
\begin{lemma}\label{lem:correction}
    The function $h_i$ satisfies
    \begin{align*}
        \Re \left(h_i(x)\right)&= |\hat{g}_i(x)|\\
        \supp \hat{h}_i&\subset \mathbf{Z}_{\leqslant 0}\\
        \lVert h_i\rVert_2&\leqslant  2\lVert \hat{g}_i\rVert_2.
        \end{align*}
    The function $Q_i$ satisfies
    \begin{align*}
        \supp \hat{Q}_i&\subset \mathbf{Z}_{\leqslant 0}\\
        |Q_i(x)|&\leqslant 1\\
        |1-Q_i(x)|&\leqslant |h_i(x)|.
        \end{align*}
\end{lemma}
\begin{proof}
Consider the Fourier expansions $|\hat{g}_i(x)|=\sum_{n\in\mathbf{Z}}c_i(n)e(nx)$ and 
    \begin{align*}
        h_i(x)=c_i(n)+2\sum_{n< 0}c_i(n)e(nx).
    \end{align*}
    One can see from this and that $c_i(n)=c_i(-n)$ that $\Re \left(h_i(x)\right)=\sum_{n\in\mathbf{Z}}c_ne(nx)= |\hat{g}_i(x)|$, that $\supp \hat{h}_i\subset \mathbf{Z}_{\leqslant 0}$ and that $\Vert h_i\rVert_2\leqslant 2\lVert \hat{g}_i\rVert_2$. Now let us consider $Q_i(x)=e^{-h_i(x)}$. That $|Q_i(x)|\leqslant 1$ follows from the bound $|Q_i(x)|=e^{-\Re h_i(x)}=e^{-|\hat{g}_i(x)|}$. To show that $|1-Q_i(x)|\leqslant|h_i(x)|$ we may take $z=h_i(x)$ in the simple inequality $|1-e^{-z}|\leqslant |z|$ which is valid for all $z\in\mathbf{C}$ with $\Re(z)\geqslant 0$.\footnote{One can prove this by observing that $1-e^{-z}=\int_0^ze^{-w}\,dw$ and that $|e^{-w}|\leqslant 1$ for all $w$ on a straight line path from $0$ to $z$ if $\Re z\geqslant 0$.} Finally, $Q_i(x)=e^{-h_i(x)}=\sum_{k\geqslant0}(-h_i(x))^k/k!$ has a Fourier expansion whose terms all have non-positive frequencies because $\supp \hat{h}_i\subset\mathbf{Z}_{\leqslant 0}$ and hence $\supp (h_i^k)^\wedge\subset\mathbf{Z}_{\leqslant 0}$. This shows that $\supp \hat{Q}_i\subset\mathbf{Z}_{\leqslant 0}$.

\end{proof}
We now use the iterative McGehee-Pigno-Smith construction
\begin{align*}
    \Phi_1(x)&=\hat{g}_1(x)\\
    \Phi_{j+1}(x)&=\hat{g}_{j+1}(x)+Q_{j+1}\Phi_j(t),
\end{align*}
and the same proof as in Lemma \ref{lem:basicMPS} may be used to deduce that $\lVert \Phi_i\rVert_\infty\leqslant 10$ for all $j\leqslant J$. A telescoping identity shows that
\begin{align*}\Phi_J(x)&=\hat{g}_J+Q_J\hat{g}_{J-1}+\dots+Q_2\dots Q_J\hat{g}_1\\
&=\sum_{j=1}^J\hat{g}_j(x)-\sum_{j=1}^{J-1}\sum_{k=j+1}^J\hat{g}_j(x)(1-Q_k)Q_{k+1}\dots Q_J.
\end{align*}
As $\lVert \Phi_J\rVert_\infty\leqslant 10$, we have
\begin{align}
    \lVert \hat{f} + E\rVert_1&\gg \langle \hat{f} +E,\Phi_J\rangle\nonumber\\
    &= \sum_{j=1}^J\langle \hat{f}, \hat{g}_j\rangle-E_1+E_2\nonumber\\
    &\geqslant J-|E_1|-|E_2|,\label{J,Eboundenergy}
\end{align}
where we used the last equation in \eqref{eq:inneprod} to get $\langle \hat{f}, \hat{g}_j\rangle\geqslant 1$ for each $j$, and where we defined \begin{align*}E_1&=\sum_{1\leqslant j<k\leqslant J}\langle \hat{f}, \hat{g}_j(1-Q_{k})Q_{k+1}\dots Q_{J}\rangle\\
E_2&= \sum_{j=1}^J\langle E, \hat{g}_jQ_{j+1}\dots Q_J\rangle.
\end{align*} We proceed by bounding $E_1, E_2$. To bound $E_2$ we simply recall that $|Q_j|\leqslant 1$ so that by Cauchy-Schwarz $|\langle E, \hat{g}_j Q_{j+1}\dots Q_J\rangle| \leqslant \lVert E\rVert_2 \lVert \hat{g}_j\rVert_2\leqslant100^{-j/2}$ as we are assuming that $\lVert E\rVert_2\leqslant |B_1|^{1/2}/K$ while $\lVert \hat{g}_j\rVert_2=|A_j|^{-1/2}\leqslant 100^{-j/2}|B_1\cup B_2|^{-1/2}K$ by our choice of the $A_j$. Hence, $|E_2|<1$.

\medskip

Bounding $E_1$ follows the classical McGehee-Pigno-Smith argument. By Lemma \ref{lem:correction}, the Fourier transform of $(1-Q_{k})Q_{k+1}\dots Q_{J}$ is supported on $\mathbf{Z}_{\leqslant 0}$ and hence
\begin{align*}
    \supp \left(\hat{g}_j(1-Q_{k})Q_{k+1}\dots Q_{J}\right)^{\wedge}\subset \mathbf{Z}\cap(-\infty,\max A_{j}].
\end{align*}
We therefore get 
\begin{align*}
&\langle \hat{f}, \hat{g}_j(1-Q_k)Q_{k+1}\dots Q_J\rangle\\
&= \langle \sum_{n\in B_1:n\leqslant \max A_j} e(nx) + \sum_{n\in B_2: n\leqslant \max A_j} e(nx), \hat{g}_j(1-Q_k)Q_{k+1}\dots Q_J\rangle\\
&\leqslant 2|(B_1\cup B_2)\cap(-\infty,\max A_j]|^{1/2}\lVert h_k\rVert_2 
\end{align*}
by Cauchy-Schwarz and Lemma \ref{lem:correction}. Using that $|(B_1\cup B_2)\cap (-\infty,\max A_j]|\leqslant 101|A_j|/100$ by our choice of the $A_j$, and that $\lVert h_k\rVert_2\leqslant 2|A_k|^{-1/2}$, we obtain the bound $5|A_j|^{1/2}|A_k|^{-1/2}=5\cdot 100^{-(k-j)/2}$ for the inner product above. In total, we get
$$|E_1|\leqslant 5\sum_{1\leqslant j<k\leqslant J} 100^{-(k-j)/2}\leqslant 5J/9.$$ Finally, we may substitute this estimate in \eqref{J,Eboundenergy} and recall that $J\gg \log K$ to obtain
\begin{align*}
    \lVert \hat{f}+E\rVert_1\gg 4J/9-1\gg \log K,
\end{align*}
which is the required conclusion.

\end{proof}

\bibliographystyle{plain}
\bibliography{referencesSum-free}

\begin{thebibliography}{10}

\bibitem{alon}
N.~Alon.
\newblock Paul {E}rd{\H{o}}s and probabilistic reasoning.
\newblock In {\em Erd\"os centennial}, volume~25 of {\em Bolyai Soc. Math. Stud.}, pages 11--33. J\'anos Bolyai Math. Soc., Budapest, 2013.

\bibitem{alonkleitman}
N.~Alon and D.~J. Kleitman.
\newblock Sum-free subsets.
\newblock In {\em A tribute to {P}aul {E}rd\H{o}s}, pages 13--26. Cambridge Univ. Press, Cambridge, 1990.

\bibitem{bilulevruzsa}
Y.~F. Bilu, V.~F. Lev, and I.~Z. Ruzsa.
\newblock Rectification principles in additive number theory.
\newblock {\em Discrete Comput. Geom.}, 19(3, Special Issue):343--353, 1998.

\bibitem{bourgain}
J.~Bourgain.
\newblock Estimates related to sumfree subsets of sets of integers.
\newblock {\em Israel J. Math.}, 97:71--92, 1997.

\bibitem{eberhard}
S.~Eberhard.
\newblock F\o lner sequences and sum-free sets.
\newblock {\em Bull. Lond. Math. Soc.}, 47(1):21--28, 2015.

\bibitem{eberhardgreenmanners}
S.~Eberhard, B.~Green, and F.~Manners.
\newblock Sets of integers with no large sum-free subset.
\newblock {\em Ann. of Math. (2)}, 180(2):621--652, 2014.

\bibitem{erdos}
P.~Erd\H{o}s.
\newblock Extremal problems in number theory.
\newblock {\em Proc. Sympos. Pure Math.}, VIII AMS:181--189, 1965.

\bibitem{green}
B.~Green.
\newblock 100 open problems.
\newblock {\em Manuscript}, https://people.maths.ox.ac.uk/greenbj/papers/open-problems.pdf.

\bibitem{greenruzsa}
Ben Green and Imre~Z. Ruzsa.
\newblock Sets with small sumset and rectification.
\newblock {\em Bull. London Math. Soc.}, 38(1):43--52, 2006.

\bibitem{guy}
R.~K. Guy.
\newblock Unsolved problems in number theory.
\newblock {\em Problem Books in Mathematics}, Springer-Verlag, New York, 2004.

\bibitem{jingwu1}
Y.~Jing and S.~Wu.
\newblock The largest {$(k,\ell)$}-sum-free subsets.
\newblock {\em Trans. Amer. Math. Soc.}, 374(7):5163--5189, 2021.

\bibitem{jingwu2}
Y.~Jing and S.~Wu.
\newblock A note on the largest sum-free sets of integers.
\newblock {\em J. Lond. Math. Soc. (2)}, 109(1):Paper No. e12819, 19, 2024.

\bibitem{konyaginl}
S.~V. Konyagin.
\newblock On the {L}ittlewood problem.
\newblock {\em Izv. Akad. Nauk SSSR Ser. Mat.}, 45(2):243--265, 463, 1981.

\bibitem{lewko}
M.~Lewko.
\newblock An improved upper bound for the sum-free subset constant.
\newblock {\em J. Integer Seq.}, 13(8):Article 10.8.3, 15, 2010.

\bibitem{malouf}
J.~L. Malouf.
\newblock Combinatorial approaches to integer sequences.
\newblock {\em ProQuest LLC, Ann Arbor, MI}, 1994.
\newblock Thesis (Ph.D.)--University of Illinois at Urbana-Champaign.

\bibitem{mcgeheepignosmith}
O.~C. McGehee, L.~Pigno, and B.~Smith.
\newblock Hardy's inequality and the {$L\sp{1}$}\ norm of exponential sums.
\newblock {\em Ann. of Math. (2)}, 113(3):613--618, 1981.

\bibitem{montgomeryvaughan}
H.~L. Montgomery and R.~C. Vaughan.
\newblock Multiplicative number theory. {I}. {C}lassical theory.
\newblock {\em Cambridge Studies in Advanced Mathematics}, Cambridge University Press, Cambridge, 2007.

\bibitem{natanson}
I.~P. Natanson.
\newblock Constructive function theory. {V}ol. {I}. {U}niform approximation.
\newblock {\em Frederick Ungar Publishing Co.}, New York, 1964.
\newblock Translated from the Russian by Alexis N. Obolensky.

\bibitem{pichorides}
S.~K. Pichorides.
\newblock On the {$L\sp{1}$}\ norm of exponential sums.
\newblock {\em Ann. Inst. Fourier (Grenoble)}, 30(2):v, 79--89, 1980.

\bibitem{rudin}
W.~Rudin.
\newblock Trigonometric series with gaps.
\newblock {\em J. Math. Mech.}, 9:203--227, 1960.

\bibitem{schur}
I.~Schur.
\newblock Über die kongruenz $x^m+y^m=z^m$ (mod. p.).
\newblock {\em Jahresbericht der Deutschen Mathematiker-Vereinigung}, 25:114--116, 1917.

\bibitem{shakan}
G.~Shakan.
\newblock On the largest sum-free subset problem in the integers.
\newblock {\em preprint, arXiv:2207.14210}, 2022.

\bibitem{taovu}
T.~Tao and V.~Vu.
\newblock Sum-free sets in groups: a survey.
\newblock {\em J. Comb.}, 8(3):541--552, 2017.

\bibitem{taovubook}
T.~Tao and V.~Vu.
\newblock Additive combinatorics.
\newblock {\em Cambridge Studies in Advanced Mathematics}, Cambridge University Press, Cambridge, 2010.

\bibitem{zygmund}
A.~Zygmund.
\newblock Trigonometric series. {V}ol. {I}, {II}.
\newblock {\em Cambridge Mathematical Library}, Cambridge University Press, Cambridge, 2002.

\end{thebibliography}
\bigskip

\noindent
{\sc Mathematical Institute, Andrew Wiles Building, University of Oxford, Radcliffe
Observatory Quarter, Woodstock Road, Oxford, OX2 6GG, UK.}\newline
\href{mailto:benjamin.bedert@magd.ox.ac.uk}{\small benjamin.bedert@maths.ox.ac.uk}

\end{document}